\documentclass[12pt]{amsart}
\usepackage{amsmath,amssymb,amsfonts,amsthm,amstext,verbatim,url}
\usepackage{graphicx}
\RequirePackage[colorlinks,citecolor=blue,urlcolor=blue]{hyperref}

\theoremstyle{plain}

\newtheorem{theorem}{Theorem}
\newtheorem{definition}[theorem]{Definition}
\newtheorem{lemma}[theorem]{Lemma}
\newtheorem{proposition}[theorem]{Proposition}
\newtheorem{corollary}[theorem]{Corollary}
\newtheorem{example}[theorem]{Example}
\newtheorem{assumption}[theorem]{Assumption}
\newtheorem{remark}[theorem]{Remark}
\numberwithin{equation}{section}
\numberwithin{theorem}{section}
\numberwithin{figure}{section}

\newcounter{mycount}
\newenvironment{romlist}{\begin{list}{\rm(\roman{mycount})}%
   {\usecounter{mycount}\labelwidth=1cm\itemsep 0pt}}{\end{list}}
\newenvironment{numlist}{\begin{list}{\arabic{mycount}.}%
   {\usecounter{mycount}\labelwidth=1cm\itemsep 0pt}}{\end{list}}
\newenvironment{letlist}{\begin{list}{\rm(\alph{mycount})}%
   {\usecounter{mycount}\labelwidth=1cm\itemsep 0pt}}{\end{list}}

\newcommand\s{\sigma}

\newcommand\ol{\overline}
\newcommand\oo{\infty}
\newcommand\De{\Delta}

\newcommand\NN{{\mathbb N}}
\newcommand\TT{{\mathbb T}}

\newcommand\sP{{\mathcal P}}
\newcommand\cA{{\mathcal A}}

\newcommand\sA{\Gamma}
\newcommand\sG{{\mathcal G}}
\newcommand\sO{{\mathcal O}}

\newcommand\sN{{\mathcal N}}
\newcommand\sB{{\mathcal B}}

\newcommand\ZZ{{\mathbb Z}}

\newcommand\la{\lambda}
\newcommand\wt{\widetilde}
\newcommand\om{\omega}

\newcommand\sm{\setminus}
\renewcommand\a{\alpha}
\newcommand\Ga{\Gr}
\newcommand\Gr{\sG}
\newcommand\Si{\Sigma}
\newcommand\si{\sigma}
\newcommand\eps{\epsilon}
\renewcommand\b{\beta}
\newcommand\g{\gamma}
\newcommand\resp{respectively}

\newcommand\de{\delta}
\newcommand\amb{\linebreak[1]}
\newcommand\lell{\ell}

\newcommand\Stab{\text{\rm Stab}}
\newcommand\Stabo{\Stab^0}
\newcommand\olmu{{\ol \mu}}
\newcommand\olG{{\ol G}}
\newcommand\olV{{\ol V}}
\newcommand\olE{{\ol E}}
\newcommand\olGa{{\sA}}

\newcommand\olF{F}
\newcommand\ole{{\ol e}}
\newcommand\vG{{\vec G}}
\newcommand\vL{\vec L}

\newcommand\vE{{\vec E}}
\newcommand\vell{\vec\ell}
\newcommand\opi{\ol\pi}
\newcommand\vpi{\vec\pi}
\newcommand\vnu{\vec\nu}
\newcommand\vmu{\vec\mu}
\newcommand\vSi{\vec\Sigma}
\newcommand\vsi{\vec\sigma}
\newcommand\q{\quad}
\newcommand\Aut{\text{\rm Aut}}
\newcommand\pd{\partial}
\newcommand\ro{{\text{\rm o}}}

\newcommand\id{\iota}
\newcommand\modl{\Lambda}
\newcommand\osA{\cA}
\newcommand\sAz{\cA}
\newcommand\fcp{finite coset property}
\newcommand\origin{\mathbf{0}}
\newcommand\pc{p_{\text{\rm c}}}

\begin{document}
\title[Strict inequalities for connective constants]
{Strict inequalities\\ for connective constants\\ of transitive graphs}
\author[Grimmett]{Geoffrey R.\ Grimmett}
\address{Statistical Laboratory, Centre for
Mathematical Sciences, Cambridge University, Wilberforce Road,
Cambridge CB3 0WB, UK} 
\email{\{g.r.grimmett, z.li\}@statslab.cam.ac.uk}
\urladdr{\url{http://www.statslab.cam.ac.uk/~grg/}}
\urladdr{\url{http://www.statslab.cam.ac.uk/~zl296/}}

\author[Li]{Zhongyang Li}

\begin{abstract}
The \emph{connective constant} of a graph is the exponential growth
rate of the number of self-avoiding walks starting at a  given vertex.
Strict inequalities are proved for connective constants of 
vertex-transitive graphs. Firstly, the 
connective constant \emph{decreases} strictly 
when the graph is replaced by a non-trivial quotient graph. 
Secondly, the connective constant \emph{increases} strictly when
a quasi-transitive family of new edges is added. These results
have the following implications for Cayley graphs.  The
connective constant of a Cayley graph decreases strictly when
a new relator is added to the group, and increases strictly when a 
non-trivial group element is declared to be a generator.

\end{abstract}

\date{14 January 2013, revised 1 January 2014 and 8 April 2014} 

\keywords{Self-avoiding walk, connective constant, regular graph, vertex-transitive graph, 
quasi-transitive graph, Cayley graph}
\subjclass[2010]{05C30, 82B20, 60K35}

\maketitle

\section{Introduction}\label{sec:intro}
A self-avoiding walk (abbreviated to SAW) is a path that revisits no vertex. Self-avoiding walks
were first introduced in the context of long-chain polymers in chemistry (see \cite{f}),
and they have been studied intensively
 since by mathematicians and physicists
(see \cite{ms}). 
If the underlying graph $G$ has some periodicity, the 
number of $n$-step SAWs
with a given starting point grows (asymptotically) exponentially as $n\to\oo$, 
with some growth rate $\mu(G)$ called the \emph{connective constant} 
of the graph. There are only few graphs $G$ for which $\mu(G)$ is known exactly,
and a substantial part of the associated literature is devoted to inequalities for such constants.
The purpose of the current work is to establish conditions under which a systematic change to $G$ results in
a \emph{strict} change to $\mu(G)$.

We have two main results for an infinite, vertex-transitive graph $G$, as follows. 
The automorphism group of $G$
is denoted by $\Aut(G)$. Precise conditions are given in the formal statements of the theorems.

\begin{numlist}
\item (Theorem \ref{si}) Let the subgroup $\sA\subseteq \Aut(G)$ act transitively on $G$, 
and let $\sAz\subseteq \sA$ be a non-trivial, normal subgroup of $\sA$
(satisfying a minor condition). 
The (directed) quotient graph $\vG=G/\sAz$
satisfies $\mu(\vG ) < \mu(G)$.

\item (Theorem \ref{thm:addedge}) Suppose new edges are added 
in such a way that the resulting graph is quasi-transitive
(subject to a certain algebraic condition). 
The connective constant of the new graph $\olG$ satisfies
$\mu(G) < \mu(\olG)$.
\end{numlist}

These inequalities have the following implications for Cayley graphs.
Let $G$ be the Cayley graph of an infinite, 
finitely generated group $\Ga$ with generator set $S$ and relator set $R$.

\begin{itemize}
\item[3.] (Corollary \ref{caleydecrease}) 
Let $G_\rho$ be the Cayley graph obtained by adding to $\Ga$ a further non-trivial relator $\rho$. 
Then $\mu(G_\rho) <\mu(G)$.

\item[4.] (Corollary \ref{thm:thm4}) Let $w$ be a non-trivial element of $\Ga$ that is not a generator,
and let $\olG_w$ be the Cayley graph obtained by declaring $w$
to be a further generator. Then $\mu(G) < \mu(\olG_w)$.
\end{itemize} 

The proofs follow partly the general approach of Kesten's proof of the
pattern theorem, see \cite{hkI} and \cite[Sect.\ 7.2]{ms}.
Any $n$-step
SAW $\pi$ in the smaller graph $G$  lifts to a SAW $\pi'$ in the larger
graph $G'$. The
idea is to show that `most' such $\pi$ contain at least $an$ sub-walks 
for which the corresponding sections of $\pi'$ may be replaced in
at least two different ways by SAWs on $G'$.
Different subsets of these sub-walks give rise to different SAWs 
on $G'$. The number of such subsets grows exponentially in $n$, and this
introduces an exponential `entropic' factor in the count of SAWs.

Whereas Kesten's proof and subsequent 
elaborations were directed at certain specific lattices, our results
apply in the general setting of vertex-transitive graphs, and they
require new algebraic and combinatorial techniques.   Indeed, the work
reported here may be the first systematic study of SAWs
on general vertex-transitive graphs.

Related questions have been considered in the contexts of percolation and
disordered systems. Consider a given model on a graph $G$, such as a percolation or an Ising/Potts model.
There is generally a singularity at some parameter-value called the `critical point'.
For percolation, the parameter in question is the \emph{density} of open sites or bonds, and for the
models of statistical physics it is the \emph{temperature}. 
It is important and useful to understand something of how the critical point  varies with the 
choice of graph. In particular, under what conditions does a systematic change in the graph
cause a \emph{strict} change in the value of  the critical point?
A general approach to this issue was presented in \cite{AG} and developed further in \cite{BGK,G94}
and \cite[Chap.\ 3]{G99}.

Turning back to SAWs, the SAW generating function has radius of convergence $1/\mu$, and is
believed (for lattice-graphs at least) to have power-law behaviour near its critical point, see \cite{ms}.
The above theorems amount to strict inequalities for the critical point as the
underlying graph $G$ varies. Despite a similarity of the problem with that of disordered systems,
the required techniques for SAWs are substantially different. 
We concentrate here on vertex-transitive graphs, and the required conditions are expressed in 
the language of algebra. 
There is another feasible approach to proving strict inequalities, namely the bridge-decomposition
method introduced by Hammersley and Welsh in \cite{HW62}, and used more
recently in various works including \cite{BGG,Li-loc} and \cite[Thm 8.2.1]{ms}.

Basic notation and  facts about
SAWs and connective constants are presented in Section \ref{sec:notation}. 
There is a large literature concerning SAWs, 
of which we mention \cite{a04,bdgs,j04,ms}.
For accounts of algebraic graph theory, see \cite{bab95,GR01,LyP,SoW}.

This paper has two companion papers, \cite{GrL1,GrL2}. In \cite{GrL1}, we prove bounds
on connective constants of vertex-transitive graphs, in particular
$\mu \ge \sqrt{\De-1}$ when $G$ is an infinite, connected, $\De$-regular,
vertex-transitive, simple graph. In \cite{GrL2}, 
we explore the effect on SAWs of the Fisher transformation, applied
to a cubic or partially cubic graph.  This work is reviewed in \cite{GrLrev}.

\section{Notation and definitions}\label{sec:notation}

All graphs considered here are connected and infinite. 
Subject to a minor exception in Section \ref{sec:cayley},
they may not
contain \emph{loops} (that is, edges both of whose endpoints are the same vertex) 
but, in certain circumstances, they are permitted to have
\emph{multiple edges} (that is, two or more edges with the same pair of endpoints). 
A graph $G=(V,E)$ is called \emph{simple} if it
has neither loops nor multiple edges. 
An edge $e$ with endpoints $u$, $v$ is written $e=\langle u,v \rangle$,
and two distinct edges with the same endpoints are called \emph{parallel}.
If $\langle u,v \rangle \in E$, we call $u$ and $v$ \emph{adjacent}
and write $u \sim v$. Let
$\pd v=\{u: u \sim v\}$ denote the set of neighbours of $v \in V$.

The \emph{degree} of vertex $v$ is the number of edges
incident to $v$. 
We assume that the vertex-degrees of a given graph $G$
are finite with supremum $\De < \oo$.  
The \emph{graph-distance} between two vertices $u$, $v$ is the number of edges
in the shortest path from $u$ to $v$, denoted $d_G(u,v)$.

The automorphism group of the graph $G=(V,E)$ is
denoted $\Aut(G)$. A subgroup $\sA \subseteq \Aut(G)$ is said to \emph{act
transitively} on $G$  (or on the vertex-set $V$)
if, for $v,w\in V$, there exists $\g \in \sA$ with $\g v=w$.
It is said to  \emph{act quasi-transitively} if there is a finite
set $W$ of vertices (called a \emph{fundamental domain})
such that, for $v \in V$, there exist
$w \in W$ and $\g \in \sA$ with $\g v =w$.
The graph is called \emph{vertex-transitive} 
(\resp, \emph{quasi-transitive}) if $\Aut(G)$ acts transitively
(\resp, quasi-transitively).  The identity element of any group is denoted by $\id$.

A \emph{walk} $w$ on $G$ is
an alternating sequence $w_0e_0w_1e_1\cdots e_{n-1} w_n$ of vertices $w_i$
and edges $e_i=\langle v_i, v_{i+1}\rangle$.
We write $|w|=n$ for the \emph{length} of $w$, that is, the number of edges in $w$.
The walk $w$ is called \emph{closed} if $w_0=w_n$. A \emph{cycle}
(or $n$-\emph{cycle}) is a closed walk $w$ with distinct edges
and $w_i\ne w_j$ for $1 \le i < j \le n$.
Thus, two parallel edges form a $2$-cycle.

Let $n \in \{1,2,\dots\}\cup\{\oo\}$. 
An \emph{$n$-step self-avoiding walk} (SAW) 
on $G$ is  a walk containing $n$ edges
no vertex of which appears more than once.
Let $\s_n(v)$ be the number of $n$-step SAWs 
 starting at $v\in V$. We are interested here in the exponential growth rate
of $\s_n(v)$. Note that, in the presence of parallel edges,
 two SAWs with identical vertex-sets but different edge-sets are considered
as distinct SAWs.

\begin{theorem}\cite{jmhII}\label{jmh}
Let $G=(V,E)$ be an infinite, connected, quasi-transitive graph with finite vertex-degrees. There exists
$\mu=\mu(G)\in[1,\oo)$, called the \emph{connective constant} of $G$,  such that
\begin{equation}\label{connconst}
\lim_{n \to \oo} \s_n(v)^{1/n} = \mu, \qquad v \in V.
\end{equation}
\end{theorem}

Subadditivity plays a key part in the proof of this theorem. It yields the inequality
\begin{equation}\label{1215}
\sup_{v \in V} \si_n(v) \ge \mu^n, \qquad n \ge 0,
\end{equation} 
which will be useful later in this paper. A further proof of Theorem \ref{jmh}
may be found  in \cite{GHP}.

See \cite[Thm 3.1]{GrL2} (and also \cite[Prop.\ 1.1]{Lac}) for an elaboration of 
Theorem \ref{jmh} in the absence of quasi-transitivity.
We note for use in Section \ref{sec:cayley} that the above notation may be extended in a natural way to
\emph{directed} graphs, and we omit the details here.
In particular, one may define the connective constant $\vmu
= \mu(\vG)$ of a \emph{directed}, quasi-transitive graph $\vG$
by \eqref{connconst} with $\s_n(v)$ replaced by the number of 
\emph{directed} $n$-step SAWs (whenever the relevant limits exist).

\section{Strict inequalities for vertex-transitive graphs}\label{sec:cayley}

\subsection{Introduction}

In this section, we present and discuss the two strict 
inequalities for connective constants of vertex-transitive graphs.
The first (Theorem \ref{thm:addedge} in Section \ref{sec:addedge}) deals with 
the effect of adding a quasi-transitive 
family of new edges, and the second (Theorem \ref{si} in Section \ref{sec:quot}) deals with quotient graphs.
Proofs
of these two principal theorems are found in Sections \ref{sec:si-proof}--\ref{sec:cayley-proof}
(readers of these proofs will be aided by familiarity with the proof
of the pattern theom, see \cite[Sect.\ 7.2]{ms}).
Two further results are presented in Section \ref{sec:two} concerning, \resp, 
the question of whether the proof of Theorem \ref{si} may be transformed into an algorithm,
and the possible removal of the assumption of normality in Theorem \ref{si}.  
Interspersed with the principal Sections \ref{sec:addedge} and \ref{sec:quot} are 
the discussions of Sections \ref{sec:dis1} and  \ref{sec:unim} which may be omitted at first reading.

The implications of the above inequalities for Cayley
graphs are presented in Section \ref{sec:cayley2}, see Corollaries \ref{caleydecrease} 
and \ref{thm:thm4}. 

Let $G=(V,E)$ be an infinite, connected, vertex-transitive, simple graph.
We assume throughout that the vertex-degree $\De$ of $G$ satisfies $\De<\oo$. 

\subsection{Quasi-transitive augmentation}\label{sec:addedge}

Let $G=(V,E)$ be as above, and let $\olG=(V, \olE)$ 
be obtained from $G$ by adding further edges, possibly in parallel
to existing edges. We shall assume that $E$ is a \emph{proper} subset of $\olE$, and that the additional
edge-set  $\olE\setminus E$ has a property of quasi-transitivity 
with respect to some automorphism-subgroup
$\osA$.  We shall require some properties of this subgroup $\osA$, and to this end
we introduce next a certain definition, followed by the main result of this subsection.

\begin{definition}\label{def:fcp}
Let $\sA \subseteq \Aut(G)$ act transitively on $G$.
A subgroup $\osA \subseteq \sA$ is said to 
\emph{have the \fcp} with \emph{root} $\rho \in V$ (in $\sA$) if there exist
$\nu_0,\nu_1,\dots,\nu_s \in \sA$, with $\nu_0=\id$ and $s<\oo$, such
that $V$ is partitioned as $\bigcup_{i=0}^s \nu_i \osA \rho$. 
It is said simply to have the \fcp\ if it has this property with some root.
\end{definition}

\begin{theorem}\label{thm:addedge}
Let $\sA$  act transitively on $G$, 
and let $\osA$ be a subgroup of $\sA$
with the \fcp. If $\osA \subseteq \Aut(\olG)$, then
$\mu(G) < \mu(\olG)$. 
\end{theorem}

The \fcp\ is somewhat technical, but it is satisfied in two cases of practical interest.
The following proposition is proved at the end of Section \ref{sec:dis1}.

\begin{proposition}\label{thm:pi}
Let $\sA$ act transitively on $G$, and $\rho \in V$. The subgroup $\osA$ of $\sA$ 
has the \fcp\ with root $\rho$ if either of the following holds.
\begin{letlist}
\item $\osA$ is a normal subgroup of $\sA$ which acts quasi-transitively on $G$.
\item The index $[\sA: \osA]$ is finite.
\end{letlist}
\end{proposition}

We ask whether the condition of Theorem \ref{thm:addedge} may be relaxed. 
More generally,
is it the case that $\mu(G) < \mu(\olG)$ whenever $G=(V,E)$ is transitive
and $\olG=(V,\olE)$ is quasi-transitive, with $E$ a proper subset of $\olE$.

\begin{example} \label{ex:sqt}

The square/triangular lattices present an elementary example of Theorem \ref{thm:addedge}
in action. Let $\ZZ^2$ denote the square lattice, and let $\sA$ be the group
of its translations.  The triangular lattice $\TT$ is obtained by adding the edge 
$e=\langle\origin,(1,1)\rangle$
and its images under $\sA$, where $\origin$ is the origin $(0,0)$. 
Since $\sA$ is a normal subgroup
of itself with the \fcp, we deduce that $\mu(\ZZ^2) < \mu(\TT)$.
This inequality is in fact an elementary consequence of known numerical bounds, 
see for example 
\cite{a04}. The current conclusion may however be extended beyond such known bounds
by use of Theorem \ref{thm:addedge} as follows.

Let $u,v\in\ZZ^2$ be linearly independent (non-zero) vectors, and let $\sAz_{u,v}$ be 
the subgroup of $\sA$ containing all translations of the form $mu+nv$
as $m$ and $n$ range over $\ZZ$.
It is easily seen that $\sAz_{u,v}$ is a normal subgroup of $\sA$ which
acts quasi-transitively, and hence $\sAz_{u,v}$ has the \fcp.

Let $w\in \ZZ^2$ satisfy $w \ne \origin$, and let $e$
be a new edge $\langle \origin ,w\rangle$. Let $\TT_{u,v,w}$
be the graph obtained from $\ZZ^2$ by adding
the set $\sAz_{u,v} e$ of all images of  $e$ under $\sAz_{u,v}$. 
The regular triangular lattice $\TT$ is retrieved by appropriate choice of $u$, $v$, $w$.

By Theorem \ref{thm:addedge}, $\mu(\ZZ^2) < \mu(\TT_{u,v,w})$. 
The corresponding percolation theorem states that the critical probabilities
satisfy $\pc(\ZZ^2) < \pc(\TT')$ for certain $\TT'$ obtained
by enhancing $\ZZ^2$; see \cite{AG} and \cite[Chap.\ 3]{G99}. 
This requires less symmetry on the distribution of new edges, and
it suffices that there exists $M<\oo$ such that every vertex is within distance 
$M$ of a new edge. 
\end{example}

The conclusion of Theorem \ref{thm:addedge} is generally invalid if $G$ is assumed only quasi-transitive.
Consider, for example, the pair $G$, $\olG$ of Figure \ref{fig:qusi},
each of which has connective constant $1$.

\begin{figure}[htbp]
\centerline{\includegraphics[width=0.7\textwidth]{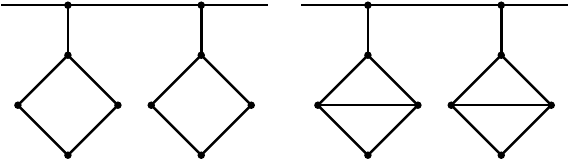}}
   \caption{The pair $G$, $\olG$. The patterns are extended infinitely in both directions. 
Each graph is quasi-transitive with connective constant $1$, 
and the second is obtained
from the first by the systematic addition of edges.}
\label{fig:qusi}
\end{figure}

\subsection{Algebraic discussion}\label{sec:dis1}

This section contains  discussion of certain algebraic facts
related to the \fcp, and it may be omitted at first reading.

The assumption of normality will recur in this paper, and we recall
the following \lq\lq standard facts".

\begin{remark}\label{rem3}
Let $\sA$ act transitively on the infinite graph $G=(V,E)$.
A partition $\sP$ of $V$is called $\sA$-\emph{invariant} if,
for $u, v\in V$ belonging to the same set of the partition,
and for $\g \in \sA$, $\g u$ and $\g v$ belong to the same set of the partition.

Let $\osA$ be a subgroup of $\sA$. The orbits
of $\osA$ form a  partition $\sP(\osA)$ of $V$. If $\osA$
is a normal subgroup, the partition $\sP(\osA)$ is $\sA$-invariant. 
If $\sP(\osA)$ is 
$\sA$-invariant, there exists a normal subgroup $\sN$ of $\sA$
such that  $\sP(\sN) = \sP(\osA)$.
In the latter case, $\sN$ may be taken as the normal closure
of $\osA$, that is, the intersection of
all normal subgroups of $\sA$ containing $\osA$.
\end{remark}

The proof of Proposition \ref{thm:pi}(b) makes use
of the relationship between the index of  $\osA$ and the number of its
orbits when acting on $V$. 
Consider $\sA$ as a group acting on $V$ with orbit-set denoted by $V/\sA$. It is said to
act \emph{freely} if every stabilizer is trivial, that is, if 
\begin{equation}\label{stab0}
\Stab_v:=\{\g\in \sA:\g v=v\}
\end{equation}
satisfies
\begin{equation}\label{stab1}
\Stab_v=\{\id\},\qquad v \in V.
\end{equation}
We abuse notation by saying that $\sA\sm \osA$ acts freely on $V$ if 
\begin{equation}\label{stab2}
\Stab_v \subseteq \osA, \qquad v \in V.
\end{equation}
The proof of the next proposition is at the end of this subsection.

\begin{proposition}\label{prop:qt}
Let $\sA$ act transitively on the countable set $V$, and let $\osA$ be a subgroup of $\sA$.
Then 
\begin{equation}\label{index}
|V/\osA| \le [\sA : \osA].
\end{equation}
If $|V/\osA| < \oo$, equality holds in \eqref{index} if and only if $\sA\sm\osA$ acts
freely on the set $V$.
\end{proposition}

We make three comments concerning Theorem \ref{thm:pi}.
\begin{numlist}

\item Let $\osA \subseteq \Aut(\olG)$ act quasi-transitively
on $G$, and suppose there exists
$\sA\subseteq \Aut(G)$
such that $\sAz \subseteq \sA$, $\sA$ acts transitively on $G$, and
$\sA \setminus \osA$ acts freely on $V$.
We have by Proposition \ref{prop:qt} that $[\sA: \osA] = |V/\osA| <\oo$, whence $\osA$
has the \fcp\  (with arbitrary root) by Theorem \ref{thm:pi}(b).

\item By Remark \ref{rem3}, 
condition (a) of Theorem \ref{thm:pi} may be replaced by the 
apparently weaker assumption that
$\osA$ is a subgroup of $\sA$ acting quasi-transitively on $G$,
whose orbits form a $\sA$-invariant partition of $V$.

\item The \emph{normal core} of $\osA$ is the intersection of the conjugate
subgroups of $\osA$. If $[\sA:\osA] < \oo$, the normal core $\sN$ of $\osA$
satisfies $[\sA:\sN]<\oo$ (see \cite[1.6.9]{Rob}).
By Proposition \ref{prop:qt},  $\sN$ acts quasi-transitively on $V$.

\end{numlist}

This section closes with the proofs of Propositions \ref{thm:pi} and \ref{prop:qt}.

\begin{proof}[Proof of Proposition \ref{thm:pi}]
(a) Suppose $\osA$ is a normal subgroup of $\sA$ acting quasi-transitively on $G$.
With $\rho \in V$, the orbits of $\osA$ may be written as 
$\osA \nu_i \rho$ for suitable $\nu_i \in \sA$
and $0 \le i \le s < \oo$. 
Since $\osA$ is normal, $\osA \nu_i = \nu_i \osA$, and the claim follows.

\smallskip
\noindent
(b) Suppose $[\sA:\osA]<\oo$. The \fcp\ holds with the $\nu_i$
chosen so that the  $\nu_i \osA$ are the left cosets of $\osA$ in $\sA$.
\end{proof}

\begin{proof}[Proof of Proposition \ref{prop:qt}]
Fix  $v_0 \in V$, and choose $v_1,v_2\dots$ such
that $\osA v_0, \osA v_1, \dots$ 
are the (distinct) orbits of $\osA$. Let $\g_0 = \id$ and, for $i \ge 1$, find
$\g_i \in \sA$ such that $\g_i v_0 = v_i$.
Let $\sA/\osA$ denote the
set of right cosets of $\osA$, and let $\phi: V/\osA \to \sA/\osA$
be given by $\phi(\osA v_i) = \osA  \g_i$. We check next
that $\phi$ is an injection.

It suffices to show that $i=j$ whenever $\osA  \g_i = \osA  \g_j$.
Suppose $\osA  \g_i = \osA  \g_j$. There exists $\a\in\osA $ such
that $\g_i = \a \g_j$. Therefore,
$$
v_i = \g_i v_0 = \a \g_j v_0 = \a v_j,
$$
whence $\osA  v_i = \osA  v_j$ and $i=j$.
Thus, $\phi$ is an injection, and \eqref{index} follows.

Suppose $|V/\osA | < \oo$, and write $U:=\bigcup_{i=0}^\oo \osA  \g_i$. 
Equality holds in \eqref{index}
if and only if $\phi$ is a surjection, which is to say that $\sA = U$.
Assume there exists $\rho \in \sA\sm U$.
Find $j$ such that $\rho v_0 \in \osA  v_j$, say $\rho v_0 = \a v_j$ with $\a \in \osA $.
Then $\a^{-1}\rho \g_j^{-1} \in \Stab_{v_j}$. If $\a^{-1}\rho \g_j^{-1} \in \osA $,
then $\rho \in \osA  \g_j$ in contradiction of the assumption $\rho \notin U$.
Therefore, $\a^{-1}\rho \g_j^{-1} \in \sA\sm\osA $, and $\sA\sm\osA $ does not
act freely.

Suppose conversely that there exist $\rho \in \sA\sm \osA $ and $v \in V$ with
$\rho v = v$. Set $w_0=v$ and find $w_i \in V$ such that the (distinct) orbits of
$\osA $ are $\osA w_0, \osA  w_1,\dots$. Let $\g_0=\id$ and,
for $i \ge 1$, find $\g_i\in \sA$ such that
$\g_i w_0 = w_i$.  If $\rho \in \osA  \g_i$, then $\rho=\a \g_i$ for some $\a\in\osA $,
so that $w_0 = \rho w_0 = \a\g_i w_0 = \a w_i$. This implies that $i=0$, and hence $\rho\in \osA $,
a contradiction. Therefore, $\rho \in \sA \sm U$ where $U = \bigcup_i \osA  \g_i$.
It follows as above that $|V/\osA | < [\sA : \osA ]$.
\end{proof}

\subsection{Quotient graphs}\label{sec:quot}

Let $\sA$ be a subgroup of the automorphism group $\Aut(G)$ that
acts transitively, and let $\sAz$ be a subgroup of $\sA$.
We denote by $\vG =(\olV,\vE)$ the (directed) quotient graph 
$G/\sAz$ constructed as follows. Let $\approx$ be the equivalence relation on $V$ given by
$v_1 \approx v_2$ if and only if there exists 
$\a\in\sAz$ with $\a v_1=v_2$. The vertex-set $\olV$ 
comprises the equivalence classes of $(V,\approx)$, that is, the orbits $\ol v := \sAz v$ 
as $v$ ranges over $V$. For $v,w \in V$, 
we place $|\pd v \cap \ol w|$ directed edges
from $\ol v$ to $\ol w$ (if $\ol v = \ol w$,
these edges are directed loops), and we write $\ol v \sim \ol w$ if $|\pd v \cap \ol w| \ge 1$
and $\ol v\ne \ol w$ (thus $\ol v \sim \ol w$ if and only if $\ol w \sim \ol v$).  
By the next lemma, the number $|\pd v \cap \ol w|$
is independent of the choice
of $v \in\ol v$.

\begin{lemma}\label{lem:welld}
Let $\sAz$ be a subgroup of $\sA$, and let $\ol v,\ol w \in \olV$.
The number $|\pd v\cap\ol w|$ is independent of the representative $v \in  \ol{v}$.
\end{lemma}

\begin{proof}
Let $v,v'\in\ol{v}$ and $v'\neq v$. 
Choose $\a\in\sAz$ such that $\a v=v'$. 
Then $x\in\pd v\cap\ol w$
is mapped to $\a x\in\pd v'\cap\ol w$, whence $\a$ acts
as an injection from $\pd v \cap \ol w$ to $\pd v' \cap \ol w$.
Therefore, $|\pd v \cap \ol w| \le |\pd v' \cap \ol w|$, and the claim follows
by symmetry.
\end{proof}

 We call $\sAz$ \emph{symmetric} if
\begin{equation}\label{sym}
|\pd v \cap \ol w| = |\pd w\cap \ol v|, \qquad v,w \in V.
\end{equation}
Two sufficient conditions for symmetry are presented in the forthcoming Lemma \ref{prop:welld}.

We recall that the orbits of $\sAz$ are
invariant under $\sA$ if and only if they are the orbits
of some normal subgroup of $\sA$ (see Remark \ref{rem3}).
\emph{Assume henceforth that $\sAz$ is a normal subgroup of $\sA$.}
It is standard that $\a\in\sA$ acts on $\vG$ 
by $\a(\sAz v) = \sAz (\a v)$, and 
that $\sA$ acts transitively on $\vG$.
Furthermore, for $v,w\in V$,
\begin{equation}\label{g703}
\ol v = \ol w\q \Leftrightarrow \q \forall \g\in\sA, \ \ol{\g v} = \ol{\g w}.
\end{equation}

Any walk $\pi$ on $G$ induces a (directed) walk $\vpi$ on $\vG$, and we say that
$\pi$ \emph{projects} onto $\vpi$.
For a walk $\vpi$ on $\vG$, there exists a walk $\pi$ on $G$ that
projects onto $\vpi$, and we  say that
$\vpi$ \emph{lifts} to $\pi$. There are generally
many choices for such $\pi$, and we fix such a choice as follows.
For  $\ol v_1, \ol v_2 \in \olV$,
we label the $N(v_1,v_2) :=|\pd v_1 \cap \ol v_2|$ directed edges 
from $\ol v_1$ to $\ol v_2$ in
a fixed but arbitrary way with the integers $1,2,\dots, N(v_1,v_2)$.
For $v \in \ol v_1$, we label similarly the edges from $v$ to
vertices in the set $\ol v_2$. For $v \in \ol v$, any $\vpi$ from $\ol v$
lifts to a unique
$\pi$ from $v$ that conserves edge-labellings, and thus walks from a given $v$ on $G$ are in
one--one correspondence with walks from $\ol v$ on $\vG$.
Since a SAW on $\vG$ lifts to a SAW on $G$, $\vmu:=\mu(\vG)$
satisfies $\vmu \le \mu(G)$.
Our task in this subsection is to identify sufficient conditions
for the strict inequality $\vmu < \mu(G)$.

We introduce next the so-called \emph{type} of the subgroup $\sAz$,
in terms of the length of
the shortest SAW of $G$ with endpoints in the same orbit.
Let $v\in V$, and let $w\ne v$ satisfy: 
$\ol w = \ol v$ and $d_G(v,w)$ is minimal
with this property. 
By the transitive action of $\sA$,
\begin{equation}\label{g708}
d_G(x,y) \ge d_G(v,w), \qquad x \ne y,\ \ol x = \ol y.
\end{equation}
We say that $v$ is of:
\begin{align*}
&\text{type 1\q if $d_G(v,w)=1$,}\\
&\text{type 2\q if $d_G(v,w)=2$,}\\
&\text{type 3\q if $d_G(v,w)\ge 3$.}
\end{align*}
By \eqref{g703},  every vertex has the same type, 
and thus we shall speak of the \emph{type} of  $\sAz$.

There follows the main theorem of this section.  A group is called \emph{trivial}
if it comprises the identity  $\id$ only. An automorphism $\b$ is said to \emph{fix}
a vertex $w$ if $\b w=w$.

\begin{theorem}\label{si}
Let $\sAz$ be a non-trivial, normal subgroup of $\sA$. 
The connective constant $\vmu=\mu(\vG)$ satisfies $\vmu < \mu(G)$ if: either
\begin{letlist}
\item the type of $\sAz$ is  $1$ or $3$, or
\item $\sAz$ has type $2$ and either of the following holds.
\begin{romlist}
\item $G$ contains  a SAW $v_0,w,v'$ satisfying $\ol v_0 = \ol v'$ and $|\pd v_0 \cap \ol w| \ge 2$, 
\item $G$ contains a SAW $v_0\,(=w_0),w_1,w_2,\dots,w_l\,(=v')$ satisfying 
$\ol v_0 = \ol v'$, $\ol w_i \ne \ol w_j$ for $0 \le i < j < l$, and furthermore $v' = \b v_0$ for
some $\b \in \sAz$ which fixes no $w_i$.
\end{romlist}
\end{letlist}
\end{theorem} 

In the special case when $\mu(G)=1$, by
\cite[Thm 1.1]{GrL1} $G$  has degree $2$ and is therefore
the line $\ZZ$. It is easily seen that $\olV$ is finite, so that $\vmu = 0$.

The conclusion of Theorem \ref{si} is
generally invalid for \emph{quasi-transitive} graphs.  
Consider, for example, the graph $G$ of Figure \ref{fig:reflect},
with $\sAz = \{\id,\rho\}$ where $\rho$
is reflection in the horizontal axis.
Both $G$ and its quotient graph have connective constant $1$.

\begin{figure}[htbp]
\centering
\includegraphics*{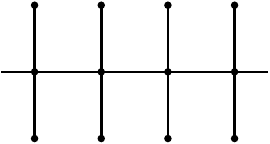}
\caption{The pattern is extended infinitely in both directions.}
\label{fig:reflect}
\end{figure}

Theorem \ref{si} is related to but distinct from the strict
inequalities of \cite{BGG,j04,SW} and \cite[Thm 8.2.1]{ms}, where
specific examples are considered of graphs that are not vertex-transitive.

Suppose $\sAz$ has type 2.
Condition (i) of Theorem \ref{si}(b)
holds if $\sAz$ is symmetric, since $|\pd w \cap \ol v| \ge 2$.
While symmetry is sufficient for
Theorem \ref{si}, we shall see in the next example that it is not necessary. 

\begin{example}\label{ex:trees}
Conditions (i)--(ii) are necessary in the type-$2$ case,
in the sense illustrated by the following example. Let $G$ be the 
infinite $3$-regular tree with a distinguished end $\om$.
Let $\sA$ be the set of automorphisms
that preserve $\omega$, and let $\sAz$ be the normal subgroup
generated by the interchange of two children of a given vertex $v$
(and the associated relabelling of their descendants). 
The graph $\vG$ is obtained from $\ZZ$ by placing two directed
edges between consecutive integers in one direction, and one directed
edge in the reverse direction.
The type of $\sAz$ is $2$.
It is easily seen that neither (i) nor (ii) holds,
and indeed $\mu(\vG) = \mu(G) = 2$. We develop this example as follows.

Let $k\ge 0$, and let 
$\sAz_k$ be the normal subgroup 
generated by $\sAz$
together with
the map that shifts $v$ to its ancestor $k$ generations earlier. 
Note that $\sAz_k$ has type $2$ for $k \ne 1$.
The case of $\sAz_1$ is trivial since $\vG$ has a
unique vertex. 

We have that $\sAz_2$ is symmetric, and condition (i) of Theorem \ref{si}(b)
applies.
In contrast, $\sAz_3$ is asymmetric (and therefore non-unimodular,
see \cite{Trof} and also \cite{Neb,SoW} as well as Lemma \ref{prop:welld}), and condition (ii) applies. 
In either case, $\vmu < \mu(G)$. 
The situation is in fact trivial since $\vG$ is a directed $k$-cycle with
two directed edges clockwise and one anticlockwise. Thus
$\vmu=0$ and $\mu(G)=2$. The same argument shows
$\vmu < \mu(G)$ in the less trivial case with $G$ the direct product of $\ZZ^d$ and the tree. 
\end{example}

A further example is provided at Example \ref{solattice}.

It is sometimes convenient to work with an undirected graph derived
from $\vG$. There are two such graphs, depending on whether or not
the multiplicities of edges are retained. The first is the simple graph,
denoted $\olG_0$, derived from $\vG$ by declaring two distinct
vertices $\ol v$, $\ol w$ to be adjacent if and only if there is a directed
edge between $\ol v$ to $\ol w$ (this property is symmetric in $\ol v$, $\ol w$).

The second such graph, denoted $\olG=(\olV,\olE)$, is a multigraph derived from
$\olG_0$ by retaining the multiplicities of parallel edges of $\vG$. 
If $\sAz$ is symmetric (in that \eqref{g804} holds), 
we obtain $\olG$ by placing $|\pd v \cap \ol w|$ (undirected) parallel edges
between each distinct pair  $\ol v,\ol w \in \olV$, and adding $|\pd v \cap \ol v|$
(undirected) loops at each $\ol v \in \olV$.

The relationship between $\vmu$ and the connective constants
of the undirected graphs derived
from $\vG$ is as follows.
It is easily seen that $\mu(\olG_0) \le \vmu$. The graph
$\olG$ is defined whenever $\sAz$ is symmetric, and in this case 
$\mu(\olG) =  \vmu$.

Finally, we discuss the assumption of normality in
Theorem \ref{si}. By Remark \ref{rem3},
this assumption may be replaced by the following: 
there exists an automorphism
group $\sA$ acting transitively on $G$, and a subgroup $\sAz$
whose partition of $V$ is $\sA$-invariant. 

\subsection{Unimodularity}\label{sec:unim}

The concept of unimodularity is relevant to the type of
$\sAz$, and we discuss this briefly in this optional section. First recall the \emph{stabilizer} 
$\Stab_v$ of a vertex $v$ given
at \eqref{stab0}.
As before, we say that $\g \in \sA$ \emph{fixes} vertex $v$ if $\g \in \Stab_v$.
Note that $|\Stab_vw| < \oo$ for $v,w\in V$,  since $G$ is locally finite and
all elements of $\Stab_vw$ are at the same distance from $v$.

Let $\Stabo_v = \Stab_v \cap \sAz$.
As shown in \cite{Trof} (see also \cite{BLPS,SoW}), when viewed as a topological group
with the usual topology, $\sAz$
is unimodular  if and only if 
\begin{equation}\label{g804}
|\Stabo_u v| = |\Stabo_v u|, \qquad u,v\in V.
\end{equation} 
Since all groups considered here are subgroups of $\Aut(G)$, we may follow
\cite[Chap.\ 8]{LyP} by \emph{defining} $\sAz$ to be \emph{unimodular} if
\eqref{g804} holds.

It turns out that $\sAz$ is symmetric whenever either it is
unimodular or it has type 3.

\begin{lemma}\label{prop:welld}
\mbox{}
\begin{letlist}
\item If $\sAz$ is unimodular, then it is symmetric (in that \eqref{sym} holds).
\item If $\sAz$ has type $3$, then $|\pd v\cap \ol v|=0$ for $v \in V$,
and $|\pd v \cap \ol w|=|\pd w\cap \ol v|=1$
whenever $\ol v\sim \ol w$. In particular, $\sAz$ is symmetric.
\end{letlist}
\end{lemma}

\begin{proof}
(a) Suppose $\sAz$ is unimodular, and let $\ol v, \ol w\in \olV$ with $\ol v \ne \ol w$. 
For $x,y\in V$, let 
$$
f(x,y) = 
\begin{cases} 1 &\text{if $x \in \ol v$, $y \in \ol w$, and $x \sim y$},\\
0 &\text{otherwise}.
\end{cases}
$$
The function $f$ is invariant under the diagonal  action of $\sAz$, in that
$f(\a x, \a y) = f(x,y)$ for $\a\in\sAz$. By the mass-transport principle
as enunciated in, for example, \cite[Thm 8.7]{LyP}, 
\begin{equation*}
 \sum_{w' \in \ol w} f(v,w') = \sum_{v' \in \ol v} f(v',w) \frac{|\Stabo_{v'}w|}{|\Stabo_w v'|},
\qquad v,w \in V.
\end{equation*}
Equation \eqref{sym} follows by \eqref{g804}.

\noindent
(b) When $\sAz$ has type 3, the claim is a 
consequence of  \eqref{g708}.
\end{proof}

\subsection{Two further results}\label{sec:two}

The proof of Theorem \ref{si}
does not appear to yield an explicit non-trivial upper bound
for the ratio $\vmu/\mu(G)$. One may, however, show the following
theorem, which is subject to the same assumptions as Theorem \ref{si}.

\begin{theorem}\label{thm:approx}
Suppose there exists a real sequence $(b_n: n=1,2,\dots)$,
each term of which may be calculated in finite time, 
such that $b_n \le \mu(G)$ and $\lim_{n\to\oo}b_n= \mu(G)$.
There exists an algorithm which terminates in finite time
and, on termination, yields a number $R=R(G,\sAz)<1$ such
that $\vmu/\mu(G) \le R$.
\end{theorem}

If $\mu$ is known, we may set $b_n\equiv\mu(G)$. 
In certain other cases, such a sequence $(b_n)$ may be found; 
for example, when $G = \ZZ^d$, we may take $b_n$ to be the $n$th root
of the number
of $n$-step bridges from the origin (see \cite{HW62,hkI}).
Theorem \ref{thm:approx} is unlikely to be useful in 
practice since the algorithm in question relies on
successive enumerations of the numbers of $n$-step SAWs on 
$G$ and $\vG$.

A result similar to Theorem \ref{thm:approx} is valid in the context of Theorem \ref{thm:addedge} also.

The relationship
between the \emph{percolation} critical points of a graph $G$ and a 
version of the quotient graph $\olG $ is the topic of
a conjecture of Benjamini and Schramm \cite[Qn 1]{BenS96}. 
As observed above (see also \cite{BenS96}), it is not necessary 
for the definition of quotient graph
to assume that $\sAz$ is a \emph{normal} subgroup of $\sA$. 
However, \cite{BenS96} appears to make the additional assumption that $\sAz$
acts freely on $V$. 
This is a stronger
assumption than unimodularity.

Following \cite{BenS96}, we ask whether $\mu(\vG) < \mu(G)$ where $\vG = G/\sAz$,
under the weaker assumption that 
$\sAz$ is a non-trivial (not necessarily normal) subgroup of $\sA$ 
acting freely on $V$, such that
$\vG$ is vertex-transitive.
The proof of Theorem \ref{si}  may be adapted
to give an affirmative answer to this question under an extra condition,
as follows. An outline proof is included at the end of Section \ref{sec:si-proof}.

\begin{theorem}\label{thm:benschr}
Let $G$ be an infinite, locally finite graph on which the automorphism group $\sA$ acts
transitively. Let $\sAz$ be a non-trivial subgroup of $\sA$ acting freely on $V$,
such that the quotient
graph $\vG := G/\sAz$ is vertex-transitive.  Assume there exists $l \ge 1$ such
that $\vG$ possesses an $l$-cycle but $G$ does not. Then
$\mu(\vG) < \mu(G)$.
\end{theorem}

\section{Strict inequalities for Cayley graphs}
\label{sec:cayley2}

We turn to the special case of Cayley graphs.
Let $\Ga$ be an infinite group
with a finite set $S$ of generators, 
where $S$ is assumed symmetric in that $S=S^{-1}$, 
and the identity $\id$ satisfies $\id \notin S$. Thus $\Ga$ has a presentation as
$\Ga = \langle S \mid R\rangle$ where $R$ is a set of relators.
The Cayley graph $G=G(\Ga,S)$ is the simple graph defined as follows.
The vertex-set $V$ of $G$ is the set of elements of $\Ga$. 
Distinct elements $g,h \in V$
are connected by an edge if and only if there exists $s \in S$ such that $h=gs$.
It is easily seen that $G$ is connected and vertex-transitive. 

The group $\Ga$ may be viewed as a subgroup of the automorphism
group of $G=(V,E)$ by: $\g \in \Ga$ acts on $V$ by  $v \mapsto \g v$. Thus, $\Ga$ acts transitively.

The reader is reminded of an elementary property of Cayley graphs, namely that
$\Ga$ acts freely on $V$.
This is seen as follows.
Suppose $\g \in \Stab_v$. For $w \in V$, there exists $h \in \Ga$ with $w=vh$,
so that $\g$ fixes every $w \in V$, whence $\g=\id$. 
See \cite{bab95} for a general account of
the theory of Cayley graphs, and \cite[Sect.\ 3.4]{LyP} for a brief account. 

Suppose a product $s_1s_2\cdots s_l$ of generators is a relator. 
The relation $s_1s_2\cdots s_l=\id$ corresponds to 
the closed walk $(\id,s_1, s_1s_2,\dots, s_1 s_2\cdots s_l=\id)$
of $G$ passing through the identity $\id$. 
Consider now the effect of adding a further relator.
Let $s_1,s_2,\dots,s_l \in S$ be such that 
$\rho := s_1s_2\cdots s_l \ne \id$, and write
$\Ga_\rho = \langle S \mid R\cup\{\rho\}\rangle$.
Then $\Ga_\rho$ is isomorphic to the quotient group $\Ga/\sN$ where
$\sN$ is the normal subgroup of $\Ga$ generated by $\rho$.
Let $G(\Ga_\rho,S)$ be the  Cayley graph of $\Ga_\rho$.

\begin{corollary}\label{caleydecrease}
Let $G=G(\Ga,S)$ be the Cayley graph of
the infinite, finitely-presented group $\Ga = \langle S\mid R\rangle$.
Let $\rho \in \Ga$, $\rho \notin R \cup \{\id\}$,
and let $G_\rho = G(\Ga_\rho,S)$.
The connective constants of $G$ and $G_\rho$ satisfy $\mu(G_\rho) <\mu(G)$.
\end{corollary}

A complementary result may be found at \cite[Thm 2]{Li-loc}.

\begin{proof}
The Cayley graph  $G_\rho$ is obtained from the quotient graph $G/\Ga_\rho$
by replacing every set of parallel edges by a single edge.
Since $\Ga_\rho$ acts freely on $V$, it is unimodular, and
the claim follows by Lemma \ref{prop:welld} and Theorem \ref{si}
(see also the statement about symmetry before Example \ref{ex:trees}).
\end{proof}

\begin{figure}[htbp]
\centering
\centering\includegraphics[width=0.5\textwidth]{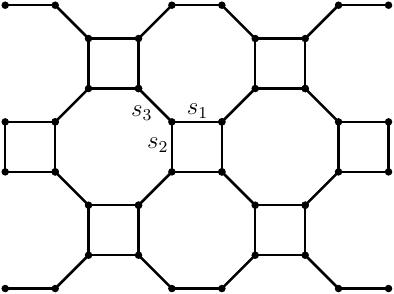}
\caption{The square/octagon lattice, also denoted $(4,8^2)$.}
\label{fig:so}
\end{figure}

\begin{example}\label{solattice}
The \emph{square/octagon} lattice, otherwise known as the Archimedean
lattice $(4,8^2)$, is illustrated in Figure \ref{fig:so}. It is the 
Cayley graph of the group with generator set 
$S=\{s_1,s_2,s_3\}$ and relators $\{s_1^2, s_2^2, s_3^2, 
s_1s_2s_1s_2,s_1s_3s_2s_3s_1s_3s_2s_3\}$.
The horizontal edges correspond to $s_1$, the vertical edges 
to $s_2$, and the other edges to $s_3$. Adding the further
relator $s_2s_3s_2s_3$, we obtain a graph isomorphic
to the ladder graph of
Figure \ref{fig:ladder-hex}, whose connective constant is the golden mean $\phi:=\frac12(\sqrt{5}+1)$. 

By Corollary \ref{caleydecrease},
the connective constant $\mu$ of the square/octagon lattice is strictly greater than 
$\phi = 1.618\dots$. The best lower bound currently known appears to
be $\mu > 1.804\dots$, see \cite{j04}.

We ask whether $\mu(G) \ge \phi$ for all infinite, vertex-transitive, simple,
cubic graphs $G$, see \cite{GrL1,GrL2}.
\end{example}

\begin{figure}[htb]
\centering
\includegraphics[width=0.4\textwidth]{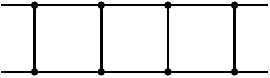}
\caption{The  (doubly-infinite) ladder graph has connective constant
the golden mean $\phi:=\frac12(\sqrt{5}+1)$; see \cite[p.\ 184]{AJ90} and \cite{GrL2}.}
\label{fig:ladder-hex}
\end{figure}

Our second inequality for Cayley graphs concerns the addition of a
generator. Let $\Ga =\langle S \mid R\rangle$ be a finitely-generated
group as above,
and let $w$ be a group element 
satisfying $w \notin S$ and $w \ne \id$. Let $\ol\Ga_w =
\langle S\cup\{w,w^{-1}\} \mid R\rangle$. 

\begin{corollary}\label{thm:thm4}
The connective constants of the Cayley graphs $G$, $\olG_w$ of
the above presentations of the groups $\Ga$, $\ol\Ga_w$ satisfy $\mu(G) < \mu(\olG_w)$.
\end{corollary}  

\begin{proof}
This is an immediate corollary of Theorem \ref{thm:addedge}, on noting
that $\Ga$ acts transitively on both $G$ and $\olG_w$.
\end{proof}

In the special case
when $\mu(G)=1$, we have  that $G=\ZZ$ (as noted after Theorem \ref{si}).
Therefore, $\olG_w$ has degree either $3$ or $4$, whence $\mu(\olG_w)
\ge \sqrt 2$ by \cite[Thm 1.1]{GrL1}.

\begin{example}
Consider $\ZZ^2$ as the Cayley graph
of the abelian group with $S = \{a,b\}$
and $R=\{aba^{-1}b^{-1}\}$. Adding the generator $ab$
(and its inverse) amounts to adding a diagonal to each square of $\ZZ^2$.
Recall Example \ref{ex:sqt}. One may easily construct more interesting examples based on, 
for example, the square/octagon lattice of Figure \ref{fig:so}.
\end{example}

For both Corollary \ref{caleydecrease} and Corollary \ref{thm:thm4}, there exists an
algorithm which, under a certain condition,
terminates in finite time and, on termination, yields an explicit non-trivial
bound for the ratio of the connective constants under consideration.
This holds just as in Theorem \ref{thm:approx}, and the details
are omitted.

\section{Proofs of Theorems \ref{si} and \ref{thm:approx}}\label{sec:si-proof}

The proof of Theorem \ref{si} is
inspired by Kesten's proof of the pattern theorem in \cite{hkI} 
(see also \cite[Sect.\ 7.2]{ms}), and a
familiarity with the latter proof will aid the reader 
of the following proofs. The overall shape of the argument
from \cite{hkI} recurs more than once in this paper; at later occurrences 
we shall outline any necessary adaptation rather
than attempt to systematize the method. 
We begin with an elementary lemma. 

\begin{lemma}\label{lem:norm}
Let $v\in V$, and let $\sAz$ be a normal subgroup of $\sA$. 
We have that $\sAz\subseteq \Stab_v$ if and only if $\sAz$ is trivial.
\end{lemma}

\begin{proof}
If $\sAz$ is trivial then $\sAz = \{\id\} \subseteq \Stab_v$. 
Conversely, suppose  $\sAz\subseteq \Stab_v$ and let $w\in V$. Since $G$ is vertex-transitive, 
there exists $\g\in \sA$ such that $\g v=w$. Since $\sAz$ is normal,
\begin{equation*}
\sAz w =\sAz \g v=\g\sAz v=\{\g v\}= \{w\}.
\end{equation*}
Therefore, $\sAz \subseteq \Stab_w$ for all $w \in V$, and hence $\sAz$ is trivial.
\end{proof}

Let $G$, $\sAz\subseteq \sA$, $\vG =G/\sAz$, etc,  be given as for Theorem \ref{si}, 
and fix $v_0\in V$.
Let $\Sigma_n$ (\resp, $\vSi_n$) be the set of $n$-step SAWs of $G$ 
(\resp, $\vG $) starting from $v_0$ (\resp, $\ol{v}_0$), and write  
$\sigma_n=|\Si_n|$ (\resp, $\vec\sigma_n=|\vSi_n|$).  

\begin{assumption}\label{assum}
We assume henceforth that either condition (a) of Theorem \ref{si} holds, 
or condition (b)(i).
An explanation of the sufficiency of condition (b)(ii) for type $2$  is given 
in Remark \ref{rem2} at the end of this proof.
\end{assumption}

Any walk $\pi$ on $G$ induces a walk $\vpi$ on $\vG$, and we say that
$\pi$ \emph{projects} onto $\vpi$.
For $\vpi \in \vSi_n$, there exists a SAW $\pi\in \Si_n$ that
projects onto $\vpi$, and we  say that
$\vpi$ \emph{lifts} to $\pi$. There are generally
many choices for such $\pi$, and we fix such a choice as 
explained after Lemma \ref{lem:welld}.
Let $\mu=\mu(G)$ and $\vmu =\mu(\vG )$. By the above remarks,
$\vmu \leq \mu$.

The idea of the proof is to replace certain sub-walks
of $\vpi\in\vSi_n$ by new walks that lift to SAWs on $G$. Such
replacements are given in terms of a certain SAW on $G$ that we introduce in
the next paragraph. We shall show that
`most' $n$-step SAWs on $\vG$ contain at least $An$ sub-walks with the
special property that they lift to at least two SAWs on $G$. This introduces a factor of $2^{An}$ in
the relative counts of SAWs on $\vG$ and $G$, and hence the required strict inequality
$\vmu < \mu$. The word `most' must be interpreted in terms of an exponential
growth rate, and the constant $A>0$ chosen appropriately.
The principal point  of divergence of the current proof from that of the pattern theorem
lies in the multiplicity of liftings: whereas the pattern theorem is combinatorial
at this point, the current proof is algebraic.

Let $\ell_{v_0}$ be a shortest
SAW from $v_0$ to some $w \ne v_0$ satisfying $\ol{v}_0=\ol{w}$. 
Such a walk exists by Lemma \ref{lem:norm} and the non-triviality of $\sAz$.
We consider $\ell_{v_0}$
as a directed walk from $v_0$ to $w$.
Let $\vell_{v_0}$ be the
projection of $\ell_{v_0}$ in $\vG$. Thus: if $G$ has type 1, 
$\vell_{v_0}$ is a loop;
if type 2, $\vell_{v_0}$ traverses a directed edge from $\ol v_0$ 
to $\ol w$, and 
then returns along a directed edge from $\ol w$ to $\ol v_0$;
if type 3, $\vell_{v_0}$ is a (directed) cycle of $\vG$ of length $3$ or more.

When $\sAz$ has type 2, we take $\ell_{v_0}$ to be the
path given in condition (b)(i) of the theorem.

For  $w\in V$, let $L_w$ (or $L(w)$) denote the set of all walks in $G$ 
starting from $w$ which are images of $\ell_{v_0}$ under elements of $\sA$, 
and let $\vL_{\ol w}$ (or $\vL(\ol w)$) be the corresponding collection of walks in $\vG$. 
Let $L= \bigcup\{L_w: w \in V\}$ and
$\vL= \bigcup\{\vL_{\ol w}: \ol w \in \ol V\}$.
All walks in $L$ have equal 
length,
written $\modl $. Thus, $\modl $ is the number
of distinct vertices in $\vell$. We shall refer to elements of $\vL$
as \emph{cycles}, and we shall use later the fact that the graph $G$, `decorated' with
members of $L$, is preserved under the action of $\sA$.

For $u \in V$ and a positive integer $r$, we define the 
\emph{ball} $B_r(u)=\{v\in V: d_G(u,v) \le r\}$.
Since  $G$ is  $\Delta$-regular,
\begin{equation}\label{200}
|B_r(u)|\leq 1+\Delta+\Delta(\Delta-1)+\dots+\Delta(\Delta-1)^{r-1}< \Delta^{r+1}.
\end{equation}

Let $\vpi$ be an $r$-step  SAW on $\vG $, 
and write $\vpi_0,\vpi_1,\dots,\vpi_r$ for the
vertices traversed by $\vpi$ in order. 
For  $k\in\mathbb{N}$, we say that $E_k=E_k(\vpi)$ occurs at the $j$th step 
of $\vpi$ if there exists $\vell \in \vL_{\vpi_j}$, denoted $\vell = \vell(\vpi_j)$,
such that at least $k$ vertices of
$\vell$
are visited by $\vpi$. 

Now let $\vpi$ be an $n$-step SAW on $\vG$. Let $j,m\in\NN$, and let $\vnu$
be the sub-SAW of $\vpi$ from $\vpi_{j-m}$ to $\vpi_{j+m}$.
(If $j-m<0$ or $j+m>n$, a modification must be made in this definition: 
if $j-m<0$, we take $\vnu$ from $\vpi_0$ to $\vpi_{j+m}$, 
and similarly if $j+m>n$.)
We say that $E_k^m$ occurs at the $j$th step of $\vpi$ 
if $E_k(\vnu)$ occurs at the $m$th step of $\vnu$. 
(If $j-m<0$, we require that $E_k(\vnu)$ occurs at the $j$th step, 
and similarly if $j+m>n$.) 
In particular, if $E_k^m$ occurs at the $j$th step of $\vpi$, then $E_k$ occurs at the $j$th 
step of $\vpi$. 

For $r\geq 0$, let $\vsi_n(r,E_k)$ (\resp, $\vsi_n(r,E_k^m)$) 
be the number of SAWs in $\vSi_n$ 
for which $E_k$ (\resp, $E_k^m$) occurs 
at no more than $r$ different steps. Observe that, for given $n$, $r$,
the count $\vsi_n(r,E_k^m)$ is non-increasing in $m$.

It is easily seen that  
$\vsi_{m+n}(0,E_k) \le \vsi_m(0,E_k)\vsi_n(0,E_k)$.
By the subadditive limit theorem, the limit
\begin{equation}\label{g1011}
\la_k := \lim_{n\to\oo} \vsi_n(0,E_k)^{1/n}
\end{equation}
exists and satisfies
\begin{equation}\label{1012}
\la_k \le \vsi_n(0,E_k)^{1/n}, \qquad n \ge 1.
\end{equation}
Thus, $\la_k < \mu$ if and only if
\begin{equation}\label{g1103}
\exists \eps>0,\, M\in\NN\text{ such that: } 
\vsi_m(0,E_k)<[\mu(1-\epsilon)]^m \text{ for } m\ge M.
\end{equation}

Our target in the next two lemmas is to show that $\la_\modl <\mu$,
and to deduce that, for suitable $a$, $m$, 
the number of $n$ step SAWs on $\vG$ in which $E_\modl^m$ occurs 
fewer than $an$ times
grows at a smaller exponential rate than the total number of SAWs.
Once the last is proved, the proof of the theorem
is fairly immediately completed in the manner of the sketch near the start of this section.

\begin{lemma}\label{zl} 
Let $k$ satisfy $1 \le k \le \modl $ and
\begin{equation}\label{g658}
\la_k<\mu.
\end{equation}
Let $\eps$, $M$ satisfy \eqref{g1103}, and let $m \ge M$ satisfy 
\begin{equation}\label{g1104}
\vsi_m \le [\mu(1+\eps)]^m.
\end{equation}
There exist $a=a(\eps,m)>0$ and $R=R(\eps,m) \in (0,1)$ such that 
\begin{equation}
\limsup_{n\to\infty}\vsi_n(an, E_k^m)^{1/n}<R\mu.
\label{czl}
\end{equation}
\end{lemma}

\begin{proof}
Assume $k$ is such that \eqref{g658} holds, and let $\eps$, $M$, $m$ satisfy
\eqref{g1103} and \eqref{g1104}.
Since $\vsi_m(0,E_k)=\vsi_m(0, E_k^m)$, 
\begin{equation}\label{g1105}
\vsi_m(0,E_k^m)<[\mu(1-\epsilon)]^m.
\end{equation}

Let $\vpi\in\vSi_n$ and  $N=\lfloor{n/m}\rfloor$. 
If $E_k^m$ occurs at no more than $r$ steps in $\vpi$, 
then $E_k^m$ occurs at no more than $r$ of the $N$ $m$-step subwalks
\begin{equation*}
(\vpi_{(j-1)m},\vpi_{(j-1)m+1},\dots,\vpi_{jm}),\qquad 1 \le j \le  N.
\end{equation*}
Counting the number of ways in which $k$ or fewer of these subwalks can contain an 
occurrence of $E_k^m$, 
we have by \eqref{g1104} and \eqref{g1105} that
\begin{align}
\vsi_n(r,E_k^m) 
&\leq \sum_{i=0}^{r} \binom Ni (\vsi_m)^i \vsi_m(0, E_k^m)^{N-i}\vsi_{n-Nm}\label{ccp}\\
&\leq \mu^{Nm} \vsi_{n-Nm} \sum_{i=0}^{r}\binom Ni
(1+\epsilon)^{im}(1-\epsilon)^{(N-i)m}.
\nonumber
\end{align}

It suffices to show that there exist $\zeta>0$ and $t<1$, depending on
$\epsilon$, $m$ only,  such that
\begin{equation}
\vsi_n(\zeta N, E_k^m)^{{1}/{N}}<t\mu^m, \label{csp}
\end{equation}
for all sufficiently large $n$, since this yields \eqref{czl} with $0<a<\zeta/m$ and $R=t^{1/m}$. 
For $\zeta$  small and positive,
\begin{multline}\label{ccp2}
\sum_{i=0}^{\zeta N} \binom Ni (1+\epsilon)^{im}(1-\epsilon)^{(N-i)m}\\
\leq (\zeta N+1) \binom N {\zeta N} \left(\frac{1+\epsilon}{1-\epsilon}\right)^{\zeta Nm}
(1-\epsilon)^{Nm}.
\end{multline}
The $N$th root of the right side converges as $N\to\infty$ to 
\begin{equation*}
\frac{1}{\zeta^{\zeta}(1-\zeta)^{1-\zeta}}
\left(\frac{1+\epsilon}{1-\epsilon}\right)^{\zeta m}(1-\epsilon)^m,
\end{equation*}
which is strictly less than $1$ for $0<\zeta<\zeta_0$, and some $\zeta_0=
\zeta_0(\epsilon, m)>0$. 
Combining this with \eqref{ccp}, we obtain \eqref{csp} for $0<\zeta<\zeta_0$, suitable $t=t(\eps,m)$, and 
$n$ sufficiently large.
\end{proof}

\begin{lemma}\label{zll}
We have that
$\la_{\modl } < \mu$.
\end{lemma}

\begin{proof}

Since each vertex $\pi_j$ is visited by  $\pi$, 
\begin{equation}\label{g701}
\vsi_n(0, E_1)=0, \qquad n \ge 1.
\end{equation}

\noindent
\emph{Case 1. If $G$ has type $1$}, then $\modl =1$, 
and $\la_{\modl }=0$ by \eqref{g701}.

\smallskip
\noindent
\emph{Cases 2/3. 
	Assume $G$ has type $2$ or $3$}, and that the lemma is false in that 
	\begin{equation}
	\la_{\modl } =\mu.
	\label{en1}
	\end{equation}
Now, $\vsi_n(0,E_k)$ (and hence $\la_k$ also)
is non-decreasing in $k$.
By \eqref{g701}, we may choose $k$ with $1\leq k<\modl  $
such that
\begin{equation}\label{ien}
\la_k<\mu, \quad \la_{k+1} = \mu.
\end{equation}
Let $\eps$, $M$ satisfy \eqref{g1103}, and let $m \ge M$ satisfy
\eqref{g1104}. By Lemma \ref{zl}, there exists $a=a(\eps,m)>0$ such that
\begin{equation}\label{g710}
\limsup_{n\to\infty}\vsi_n(an,E_k^m)^{{1}/{n}}<\mu.
\end{equation}

Let $T_n$ be the subset of $\vSi_n$ comprising
SAWs for which $E_{k+1}$ never occurs but $E_k^m$
occurs at least $an$ times.
We have that
\begin{equation*}
|T_n|\geq \vsi_n(0,E_{k+1})-\vsi_n(an,E_k^m),
\end{equation*}
whence, by \eqref{ien} and \eqref{g710}, 
\begin{equation}
\lim_{n\to\infty}|T_n|^{{1}/{n}}=\mu.
\label{tl}
\end{equation}
Thus,  under \eqref{en1},
for \lq most' SAWs $\vpi$, there exist
many  cycles in $\ol L$ having exactly 
$k$ ($< \modl $) vertices visited by $\vpi$ and none with more than $k$
such vertices. 
The rest of the proof is devoted to showing  the existence of $S=S(\eps,m,\De,\modl )<1$ such that
\begin{equation}\label{g11055}
\limsup_{n\to\infty}|T_n|^{{1}/{n}} \le S\mu.
\end{equation}
This contradicts \eqref{tl}, and the claim follows.

The idea is as follows.
Let $\vpi\in T_n$, and consider the 
set of cycles $\vell(\vpi_j)$ with exactly $k$ vertices visited by $\vpi$. 
Where such a cycle is met by $\vpi$,
we may augment $\vpi$ with an entire copy of it. 
The ensuing transformation is not one--one, and the length of 
the resulting walk 
differs from that of $\vpi$. By selecting the
places where the new elements of $\vL$ are added, 
we shall show that the number of resulting walks exceeds $|T_n|$ by 
an exponential factor. 
It is key that such augmented walks lift to SAWs on $G$ 
while traversing cycles in $\vG $, and thus we shall contradict \eqref{tl}.

Let $\vpi\in T_n$, so that $\vpi$ contains at least $an$ occurrences of 
$E_k^m$. We can find $j_1<\dots<j_u$ 
with $u=\lfloor \kappa n\rfloor -2$ where
\begin{equation}\label{g459}
\kappa =\frac{a}{(2m+2)\Delta^{2\modl  +1}},
\end{equation}
such that 
\begin{equation}\label{cd-1}
\mbox{$E_k^m$ occurs at the $j_1$th, $j_2$th, $\dots$, $j_u$th steps of $\opi$},
\end{equation}
(and perhaps at other steps as well), and in addition
\begin{gather} 
0<j_1-m,\quad j_u+m<n,\label{cd0}\\
j_{t}+m<j_{t+1}-m,\qquad 1 \le t< u,
\label{cd1}\\
\forall \vell_t \in \ol L(\vpi_t), \text{ the $\vell_1, \vell_2, \dots, \vell_u$  
are pairwise vertex-disjoint}.
\label{cd2}
\end{gather}
Such $j_t$ may be found by the following iterative construction.
First, $j_1$ is the smallest $j>m$ such that $E_k^m$ occurs at the $j$th step of $\vpi$.
Having found $j_1,j_2,\dots,j_r$, let 
$j_{r+1}$ be the smallest $j$ such that
\begin{numlist}
\item $j_r+m < j -m$,
\item every element of $\vL (\vpi_j)$ is disjoint from every element of 
$\vL(\vpi_{j_1}),\amb \vL(\vpi_{j_2}), \dots, \vL (\vpi_{j_r})$,
\item $E_k^m$ occurs at the $j$th step of $\vpi$.
\end{numlist}
Condition 1 gives rise to the factor $2m+2$ in the denominator of \eqref{g459},
and, by \eqref{200}, condition 2 gives rise to the factor $\De^{2k+1}$
($\le \De^{2\modl +1}$).

Let $t \in \{1,2,\dots,u\}$.
Since $E_k^m$ but not $E_{k+1}$ occurs at the $j_t$th step, 
$\vpi$ visits at most $k$ vertices in each cycle of $\vL(\vpi_{j_t})$.
Let $\vL(\vpi_{j_t},k)$ be the subset
of $\vL(\vpi_{j_t})$ containing all such cycles with 
exactly $k$ vertices visited by $\vpi$, and such that these $k$ vertices lie between
$\vpi_{j_t-m}$ and $\vpi_{j_t+m}$ on $\vpi$.
Choose a specific cycle $\vell(\vpi_{j_t}) \in \vL(\vpi_{j_t},k)$.
For $t=1,2,\dots,u$, let
\begin{equation}\label{g123}
\alpha_t=\min\{i:\vpi_i \in \vell(\vpi_{j_t})\},\quad
\omega_t  =\max\{i:\vpi_i \in \vell(\vpi_{j_t})\},
\end{equation} 
so that
\begin{equation*}
j_t-m\leq \alpha_t\leq j_t\leq \omega_t\leq j_t+m,\qquad 1 \le t \le u.
\end{equation*}

We describe next the strategy for replacement of the subwalk 
$(\vpi_{\alpha_t},\allowbreak \vpi_{\a_t+1},\dots, \vpi_{\omega_t})$.  
Starting from $\vpi_{\alpha_t}$, the new walk winds once around the cycle 
$\vell := \vell(\vpi _{j_t})$; having returned to the vertex  $\vpi_{\alpha_t}$, 
it continues around the same cycle $\vell$ until it reaches 
the vertex $\vpi_{\omega_t}$. 
This new subwalk is inserted into $\vpi$ at the appropriate place.
The resulting walk $\wt{\pi}^{(t)}$ is evidently not self-avoiding in $\vG$ 
(it includes a unique cycle, namely $\vell$ in some order),
but we shall see that it lifts to a SAW $\pi_\ast^{(t)}$ on $G$. 
The precise definition and properties of $\wt\pi^{(t)}$ 
and $\pi_*^{(t)}$ are described next.

Suppose $\vpi \in T_n$ lifts to $\pi\in \Si_n$.
The initial segment $(\vpi_0, \dots, \vpi_{j_t})$ 
of $\vpi$ lifts to a  SAW of $G$
that traverses the vertices $\pi_0,\pi_1,\dots,\pi_{j_t}$.
We write $v := \pi_{j_t}$, and we shall consider graphs of type 
2 and 3 separately.

\smallskip
\noindent
\emph{Case 3. Assume that $G$ has type 3}, 
and think of $\vell$ as a rooted, directed cycle of $\vG$ with root $\ol v$. 
The cycle $\vell$  lifts to 
the SAW $\ell := \ell_v$ of $G$  traversing the vertices
$v\,(=w_0),w_1,w_2,\dots, w_{\modl }$. We have that
\begin{equation}\label{g707}
v\, (=w_0),w_1,\dots, w_{\modl -1} \text{ belong to different equivalence classes}
\end{equation}
of $(V, \approx)$, and we may choose $\b\in\sAz$ such that $w_{\modl } = \b v$.

We prove next that
\begin{equation}\label{g709}
w_r \ne \b w_r, \qquad 1 \le r < \modl.
\end{equation}
Suppose first that $w_1 = \b w_1$,
and consider the walk
$$
v\, (=w_0), w_1, \dots, w_{\modl }\, (=\b v), \b w_1\, (=w_1).
$$
By the triangle inequality,
\begin{equation}\label{g812}
\modl  = d_G(v, \b v) \le d_G(v,w_1) + d_G(\b w_1, \b v) = 2,
\end{equation}
a contradiction since $\modl  \ge 3$. 
Applying the same argument to the walk
$$
w_1,w_2,\dots,w_{\modl }\,(=\b v), \b w_1, \b w_2,
$$
we obtain by \eqref{g708} that 
$w_2 \ne \b w_2$, and \eqref{g709} follows by iteration.

Suppose 
$\vpi_{\a_t}$ lifts
to vertex $x \in V$ with $\ol x= \ol w_i$ and $i \ge 1$;
similarly, suppose $\vpi_{\om_t}=\ol w_j$.
We show next that the  replacement of the subwalk 
$(\vpi_{\alpha_t},\dots, \vpi_{\omega_t})$
lifts to some SAW of $G$. 
Find $\g \in \sAz$ such that $x = \g w_i$.

\smallskip
\noindent
\emph{Case 3.1. Suppose $i<j$.}
Consider the walk
\begin{equation}\label{g706}
x\, (=\,\g w_i), \g w_{i+1}, \dots, \g w_{\modl }\,(=\g \b v),  \g \b w_1, \dots, \g \b w_i,
\end{equation}
followed by $\g\b w_{i+1},\g\b w_{i+2},\dots,\g\b w_j$.
By \eqref{g707}, two  vertices of this walk are equal if and only if there
exists $r$ such that $1\le r < \modl $ and $w_r = \b w_r$. 
By \eqref{g709}, this does not occur.
We have proved that $(\vpi_0,\dots,\vpi_{\a_t})$,
followed by the above walk, lifts to a SAW $\nu$ on $G$.
Thus, $\wt\pi^{(t)}$ lifts to the SAW $\nu$ followed by the image of $(\pi_{\om_t},\dots,\pi_n)$
under the map that sends $\pi_{\om_t}$ to $\g\b w_j$.
Since $\vpi \in T_n$, $\wt\pi^{(t)}$  lifts to a SAW $\pi_*^{(t)}$ of $G$.

\smallskip
\noindent
\emph{Case 3.2. Suppose $i > j$.}
Consider the walk
\begin{equation}\label{g706b}
x\, (=\,\g w_i), \g w_{i-1}, \dots, \g w_{0}\,(=\g  v),  \g \b^{-1} w_{\modl -1}, \dots, \g \b^{-1} w_i,
\end{equation}
followed by $\g\b^{-1} w_{i-1},\g\b^{-1}w_{i-2},\dots,\g\b^{-1} w_j$.
By \eqref{g709}, this is a SAW on $G$, and the step is completed as above.

\smallskip
\noindent
\emph{Case 2. Assume that $G$ has type $2$,} so that $\modl =2$ and $k=1$.
The required argument is slightly different 
since \eqref{g812} is no longer a contradiction. 
Let $j_t$ be as above, and let $\vell =
\vell(\vpi_{j_t}) \in \vL(\vpi_{j_t})$ 
(with corresponding $\ell\in L$)
be a witness to the occurrence
of $E_1^m$ at $\opi_{j_t}$. As in \eqref{g123},
\begin{equation}
\begin{aligned}\label{g123+}
\alpha_t=\a_t(\vpi)&=\min\{i:\vpi_i \in \vell\},\\
\omega_t =\om_t(\vpi) &=\max\{i:\vpi_i \in \vell\}.
\end{aligned} 
\end{equation}

Let $\pi$ be the lift of $\vpi$ to a SAW in $G$ from $v_0$, and
write $v=\pi_{j_t}$.
Recall that $\ell$ is a SAW visiting $v,w,\b v$ in $G$
for some $\b=\b_t \in \sAz$ with $\b v \ne v$, and the pair $v$, $w$ are visited  
(in some order) at the $\a_t$th and $\om_t$th steps of $\pi$. 
We may assume that $\b w\neq v$, 
since otherwise $w$ is adjacent
to $\b w$, and $G$ is of type 1.
We describe next the required substitution.

\smallskip
\noindent
\emph{Case 2.1. Suppose that $\b w\neq w$.}
As in Cases 3.1 and 3.2 above, both $v,w,\b v,\b w$ 
and $w,v,\b ^{-1}w,\b^{-1}v$ 
are SAWs on $G$. We write $\pi_*^{(t)}$ for the SAW on 
$G$ obtained by replacing the 
segment  of $\pi$ between $\pi_{\a_t}$ and $\pi_{\om_t}$
by
\begin{align*}
v,w,\b v,\b w\quad &\text{if $\pi$ 
 visits $v$ before $w$},\\
	 w,v,\b ^{-1}w,
 \b^{-1}v \quad&\text{if $\pi$ visits $w$ before $v$},
 \end{align*}
and the walk after $\b w$ (\resp, $\b^{-1}v$) 
by the image under $\b$ (\resp, $\b^{-1}$) 
of $\pi$ after $w$ (\resp, $v$).

\noindent
\emph{Case 2.2.1. Suppose that $\b w=w$, 
and $\pi$ visits $w$ before $v$.} Only in the following 
shall we use the assumed condition (b)(i) of Theorem \ref{si}.
Since $|\pd v \cap \ol w|, |\pd w \cap \ol v| \ge 2$, there exist $w' \in \ol w$
and $v' \in \ol v$ such that 
\begin{equation}\label{g800}
w,v,w',v'
\end{equation}
is a SAW of $G$. Let $\g \in \sAz$ be such that $v' = \g v$.
We write $\pi_*^{(t)}$ for the SAW on $G$ obtained 
by replacing the segment of $\pi$ between $\pi_{\alpha_t}$
and $\pi_{\omega_t}$ by \eqref{g800},
and the walk after $v'$ ($= \g v$) by the image under $\g$ 
of $\pi$ after $v$.

\smallskip
\noindent
\emph{Case 2.2.2. Suppose that $\b w = w$, and 
$\pi$ visits $v$ before $w$.}
Since $|\pd v \cap \ol w| \ge 2$, there exists $w'\in \ol w$ such that
\begin{equation}\label{g801}
v,w,\b v,w'
\end{equation}
is a SAW on $G$. We find $\g \in \sAz$ such that $w'= \g w$, 
and we write $\pi_*^{(t)}$ for the SAW on $G$ obtained by replacing the segment of $\pi$ between $\pi_{\alpha_t}$ and
$\pi_{\omega_t}$ by \eqref{g801},
and the walk after $w'$ ($=\g w$) by the image 
under $\g$ of $\pi$ after $w$.

\smallskip

This ends the definitions of the required substitutions, 
and we consider next their enactment. 
Let $\delta>0$, to be chosen later, and set $s=\delta n$.
(Here and later, for simplicity of notation, 
we omit the integer-part symbols.)
Let $H=(h_1,h_2,\dots,h_s)$ be an ordered subset of
$\{j_1,j_2,\dots,j_u\}$. We shall make an appropriate substitution 
in the neighbourhood of each $\vpi_{h_t}$, by an iterative
construction.

We consider the cases $t=1,2,\dots,u$ in order. 
First, let $t=1$. If $j_1 \notin H$, we do nothing, and we set $\pi^{(1)}=\pi$.
If $j_1 \in H$, we make the appropriate substitution around
the point $\vpi_{j_1}$, and write $\pi^{(1)}$ for the 
resulting SAW on $G$.
Now let $t=2$. Once again, if $j_2 \notin H$, we do
nothing, and set $\pi^{(2)} =\pi^{(1)}$. Otherwise, let $\a_2'=\a_2(\vpi^{(1)})$,
$\om_2'=\om_2(\vpi^{(1)})$.
The sub-SAW of $\pi^{(1)}$ between steps $\alpha'_2$ and $\omega'_2$ 
is the image of $\pi$ 
between steps $\alpha_2$ and $\omega_2$ under some 
$\g'\in\sAz$.  
Therefore, we may perform operations on $\pi^{(1)}$ 
after the $\a_2'$th step
as discussed above for the $\a_2$th step of $\pi$,
obtaining thus a SAW $\pi^{(2)}$.
This process may be iterated to obtain a sequence 
$\pi^{(1)},\pi^{(2)},\dots,\pi^{(u)}$ of SAWs on $G$, and we set $\pi_*=\pi^{(u)}$.
The role of $H$ is emphasized by writing $\pi_*=\pi_*(\vpi,H)$.
It follows from the construction that $\pi_*(\vpi,H) \ne \pi_*(\vpi, H')$
if $H \ne H'$.

One small issue arises during the iteration, namely that
the projections of the SAWs $\pi^{(t)}$ are not generally 
self-avoiding.  However, by \eqref{cd0}--\eqref{cd2}, 
the subwalk of $\vpi$ under inspection at any given 
step is disjoint from all previously inspected subwalks, and
thus the current substitution is unaffected by the past. 

By \eqref{cd2} and the discussion above, $\pi_*$ is self-avoiding on $G$. 
By inspection of such $\pi_*$ and its projection,
one may reconstruct the places at which cycles have been added to $\vpi$.
The length of $\pi_*$ does not exceed $n+2\modl s$.

We estimate the
number of pairs $(\vpi,H)$ as follows. First, the number $|(\vpi,H)|$
is at least the cardinality of $T_n$ 
multiplied by the minimum number of possible choices of $H$ 
as $\vpi$ ranges over $T_n$. 
Any subset of $\{j_1,j_2,\dots,j_{u}\}$ 
with cardinality $s=\delta n$ 
may be chosen for $H$,
whence  
\begin{equation}
|(\vpi,H)| \geq |T_n|\binom{\kappa n-2}{\delta n}.
\label{lb0}
\end{equation}

We bound $|(\vpi,H)|$ above by counting the 
number of SAWs $\pi_*$ of $G$ 
with length not exceeding $n+2\modl  \delta n$,
and multiplying by an upper bound for the number of pairs $(\vpi, H)$
giving rise to a particular $\pi_*$. The number of possible choices
for $\pi_*$ is no greater than $\sum_{i=0}^{n+2\modl \delta n} \sigma_i$.
A given $\pi_*$ contains $|H|=\delta n$ 
elements of $L$. 
At the $t$th such occurrence, 
$\vpi_{h_t}$ is a point on the corresponding cycle, 
and there are no more than
$2\modl  $ different choices for $\vpi_{h_t}$. For given $\pi$ and 
$(\vpi_{h_t}: t=1,2,\dots,s)$,  there are at most 
$\left(\sum_{i=1}^{2m}\vsi_i\right)^{\delta n}$ corresponding SAWs $\vpi$
of $\vG $.  
Therefore,
\begin{equation}
|(\vpi,H)|
\leq \left(2\modl  \sum_{i=1}^{2m}\vsi_i\right)^{\delta n}
\left( \sum_{i=0}^{n+2\modl  \delta n}\sigma_i \right).
\label{ub}
\end{equation}

Let $\tau := \limsup |T_n|^{1/n}$. 
We  combine \eqref{lb0}--\eqref{ub}, take $n$th roots and the limit as $n\to\infty$, 
to obtain, by the fact that $\sigma_N^{1/N} \to \mu$,
\begin{equation*}
\tau\frac{\kappa^{\kappa}}{\delta^{\delta}(\kappa-\delta)^{\kappa-\delta}}\leq \left(2\modl  \sum_{i=1}^{2m}\vsi_i\right)^{\delta}\mu^{1+2\modl  \delta}.
\end{equation*}
There exists $Z=Z(\eps, m, \De, \modl )<\oo$ such that
$$
2\modl \mu^{2\modl  } \sum_{i=1}^{2m}\vsi_i\le Z.
$$ 
Therefore,
\begin{equation*}
\tau \leq f(\eta)^\kappa \mu,
\end{equation*}
where $f(\eta)= Z^\eta \eta^\eta(1-\eta)^{1-\eta}$ and $\eta={\delta}/{\kappa}$.
Since 
\begin{equation*}
\lim_{\eta\downarrow 0}f(\eta)=1, \quad \lim_{\eta\downarrow 0}f'(\eta)=-\infty,
\end{equation*}
we have that $f(\eta)<1$ for sufficiently small $\eta=\eta(Z)>0$,
and \eqref{g11055} follows for suitable $S<1$.
The proof is complete.
\end{proof}

\begin{proof}[Proof of Theorem \ref{si} under Assumption \ref{assum}]

By Lemma \eqref{zll}, $\la_{\modl }<\mu$.
Let $\eps$, $M$ satisfy \eqref{g1103} and let $m \ge M$ satisfy
\eqref{g1104}, with $k=\modl $.
By Lemma \ref{zl}, there exist $a=a(\eps,m)>0$   and $R=R(\eps,m)\in(0,1)$ 
such that 
\begin{equation}
\limsup_{n\to\infty}\vsi_n(an,E_{\modl }^m)^{{1}/{n}}<R\mu.
\label{pin}
\end{equation}

Let $T_n$ be the subset of $\vSi_n$ comprising 
SAWs for which $E_{\modl }^m$ occurs 
at least $an$ times. Thus
\begin{equation}\label{g1106}
|T_n| \ge \vsi_n - \vsi_n(an,E_{\modl }^m).
\end{equation}
We follow the route of the previous proof. For $\vpi\in T_n$ and $\kappa$
as in \eqref{g459}, we may find $j_1<\dots<j_u$ 
with $u=\kappa n -2$ such that
\eqref{cd-1}--\eqref{cd2} hold with $k=\modl $.

Let $\delta>0$,  set $s = \de n$, and let $\vpi\in T_n$. 
We choose a subset $H$
of $\{j_1,j_2,\dots,j_u\}$ with cardinality $s$, and we
construct a SAW $\pi_*=\pi_*(\vpi,H)$ on $G$ accordingly. 
This may be done exactly as in the previous proof \emph{if $G$
has type $2$ or $3$} to obtain as in \eqref{g11055} that
there exists $S=S(\eps,m,\De,\modl )<1$ such that
\begin{equation}\label{g1107}
\limsup_{n\to\infty}|T_n|^{1/n} \le S\mu.
\end{equation}
By \eqref{pin}--\eqref{g1106},
\begin{equation*}
\vmu = \lim_{n\to\oo} \vsi_n^{1/n} \le \max\{R,S\}\mu,
\end{equation*}
and the claim of the theorem is proved in this case.

\smallskip
\noindent
\emph{Assume $G$ has type 1.}
This case is easier than the others. Suppose $\vpi_{j_t}=\ol v$,
and $\vpi$ lifts to $\pi$ with $\pi_{j_t}=v$. We obtain 
$\pi_*^{(t)}$ by replacing $v$ by the pair $v,\g v$ for 
some $\g\in \sAz$ with $\g \ne \id$, and translating by $\g$ the 
subwalk of $\pi$ starting at $v$.  
The above counting argument yields the result.
\end{proof}

\begin{remark}\label{rem2}
Theorem \ref{si} has been proved subject to Assumption \ref{assum},
and it remains to prove it when $\sAz$ has type $2$ and condition (b)(ii)
holds. Suppose the latter holds, so that $l \ge 2$. If $l=2$,
the argument in the proof of Lemma \ref{zll} is valid (this is Case 2.1 of the proof, 
which does not use condition (i)), and the theorem follows similarly.
Assume $l \ge 3$. By vertex-transitivity, we may set $v_0 = v$, and replace
$\ell_{v_0}$ by the path of (ii), so that $\modl =l$. Case 3 of the proof of Lemma \ref{zll} may
be followed with one difference, namely that \eqref{g709} holds by the assumption that
$\b$ fixes no $w_i$.
\end{remark}

\begin{proof}[Proof of Theorem \ref{thm:approx}]
We use the notation of the previous proofs. The constants $R$, $S$ of the last proof
may be calculated explicitly in 
terms of $\eps$, $m$, $\De$, $\modl $, and so it
suffices to describe how to choose $\eps$ and $m$. 

Let $a_n = \si_n^{1/n}$, noting by \eqref{1215} that $a_n \ge \mu$
and $a_n \to \mu$. By changing $(b_n)$ if necessary,
we may assume for convenience that $(b_n)$ is a non-decreasing
sequence, so that $b_n \uparrow \mu$.

By Lemma \ref{zll},  $\la_{\modl }<\mu$.
Since $b_r(1-r^{-1}) \uparrow \mu$,
we may find the earliest  $r$ such that
$$
\vsi_r(0,E_{\modl })^{1/r} < b_r(1-r^{-1}).
$$   
Let $\eps=r^{-1}$, and let $s\ge r$ be such that
$b_s(1+\eps) \ge a_s(1+\frac12\eps)$. Since $b_s \ge b_r$,
$$
\vsi_r(0,E_{\modl })^{1/r} < b_s(1-\eps).
$$
By \eqref{1012},  $\la_{\modl } < b_s(1-\eps)$, so that
\begin{equation}\label{g1112}
\vsi_m(0,E_{\modl })^{1/m} < b_s(1-\eps)
\end{equation}
for infinitely many values of $m$.

Since $a_s \ge \mu$, we may
find the earliest $m$ such that \eqref{g1112} holds and in addition
$\vsi_m^{1/m} \le b_s(1+\eps)$.
With $\eps$ and $m$ defined thus, 
\eqref{g1104} and \eqref{g1105} are valid as required.

Each of the above computations requires a finite enumeration of
a suitable family of SAWs.
\end{proof}

\begin{proof}[Outline Proof of Theorem \ref{thm:benschr}]

Let $v_0 \in V$ with orbit $\ol v_0 = \sAz v_0$.
With $l\ge 1$ as given, let $\vell$ be a (directed)  cycle of $\vG$ of length $l$.
Since $\vG$ is vertex-transitive, we may assume $\vell$ goes through
$\ol v_0$, and thus we write $\vell_{\ol v_0}$ for $\vell$, considered
as a (directed) walk from $\ol v_0$ to $\ol v_0$. Now, $\vell_{\ol v_0}$
lifts to a walk $\ell_{v_0}$ from $v_0$ on $G$. Since every vertex of $\vell_{\ol v_0}$ other than
its endpoints are distinct, and $G$ possesses no cycle of length $l$, 
$\ell_{v_0}$ is a SAW from $v_0$ on $G$.

Let $\sB$ be a subgroup of $\Aut(\vG)$ acting transitively on $\vG$.
For $w \in V$, let $\vL(\vec w)$ be the set of images of $\vell_{\ol v_0}$ under $\sB$,
and let $L(w)$ be the set of lifts of such images. 
Since $G$ possesses no cycle of length $l$, every member of $L(w)$ is
a SAW on $G$.

We now follow the above proof of Theorem \ref{si}. No assumption is made
on the relationship between the groups $\sA$ and $\sB$. If $l \ne 2$,
the argument is as presented previously except insomuch as that 
\eqref{g709} holds (since $\sAz$ acts freely). If $l=2$,  we follow
Case 2.1 as in Remark \ref{rem2}.
\end{proof}

\section{Proof of Theorem \ref{thm:addedge}}
\label{sec:cayley-proof}

The general idea of the proof of Theorem \ref{thm:addedge} 
is similar to that of Theorem \ref{si}, but
with some differences. 
The graphs $G$, $\olG$ have the same vertex-set $V$, 
but $\olG$ possesses cycles not present in $G$.
Let $v_0, w_0 \in V$ be such that there are strictly more edges
of the form $\langle v_0,w_0\rangle$
in $\olE$ than in $E$.
Since $G$ is connected, it has a shortest path from $v_0$ to $w_0$, written
\begin{equation}\label{ccl}
v_0 f_1 v_1 f_2 v_2 \cdots f_{\modl}v_{\modl}\,(=w_0)
\quad \text{where}\quad f_t = \langle v_{t-1},v_t \rangle.
\end{equation}
The path \eqref{ccl}, followed by the new
edge $f := \langle w_0,v_0\rangle$, forms 
a cycle in $\olG$ but not in $G$. Let $C_1$ denote this cycle of $\olG$, 
and note that $C_1$ has length $\modl +1$. 
As usual, $\De$ denotes the vertex-degree of $G$.

Let $C$ be a finite set of 
vertices of $G$ with $v_0\in C$, and denote by
$\g C=\{\g z: z \in C\}$ the image of $C$ under the automorphism $\g\in\olGa$. 
Let $\pi = (\pi_0,\pi_1,\dots,\pi_n)$ be an $n$-step SAW on $G$ from $v_0$.
For $k \in \NN$, we say that $E_k(C)$ occurs at the $j$th step of $\pi$ if there exists  
$\g\in\olGa$ with
$\g v_0=\pi_{j}$ such that at least $k$ vertices of 
$\g C$ are visited by $\pi$. 
Let $\sigma_n(r,E_k(C))$ (\resp, $\s_n(r,E_*(C))$) 
be the number of $n$-step SAWs on $G$ from  $v_0$ 
in which $E_k(C)$ (\resp, $E_{|C|}(C)$) occurs at no more than $r$ steps.
When $C$ is the vertex-set of $C_1$, 
we write $E^1_k = E_k(C_1)$ for $E_k(C)$, and so on.

The main step of the current proof is the following lemma. 

\begin{lemma}\label{l2}

With $\mu=\mu(G)$, we have that
\begin{equation}\label{g8120}
\liminf_{n\to\infty}\sigma_n(0,E_2^1)^{1/n}<\mu.
\end{equation}
\end{lemma}

As in \eqref{g1011}, the $\liminf$ is in fact a limit (here and later).
An automorphism $\g \in \olGa$ acts on the edge $e = \langle x,y\rangle$ by
$\g e = \langle \g x, \g y\rangle$.
Let $\asymp$ be the equivalence relation on the edge-set of $\olG$ given by
$e \asymp e'$ if there exists $\g\in\olGa$ 
such that $\g e=e'$.   
We write $\ole$ for the equivalence class containing
edge $e$, and $\{\ole: e\in E\}$ for the set of such classes. 
Since $\olGa$ acts transitively on $\olG$, for every edge $e$ and vertex $v$,
there exists a member of $\ole$ that is incident to $v$.

The proof of Lemma \ref{l2} will make use of the intermediate Lemmas \ref{lem:Ch}--\ref{l1}.

\begin{lemma}\label{lem:Ch}
Assume that
\begin{equation}\label{g813}
\lim_{n\to\oo} \sigma_n(0,E^1_2)^{{1}/{n}} = \mu.
\end{equation}
There exist $h \ge 2$ and a sequence 
$(\ole_r: r =1,2,\dots,h-1)$
of distinct equivalence classes of edges of $G$, with $\ole_1 = \ol f_1$, such that the following holds.

For $r \ge 1$, let 
$$
\olF_r =\ole_{1} \cup \ole_{2} \cup \dots \cup \ole_{r},
$$
and let $G_r$ be the graph obtained from $G$ by the deletion of 
$\olF_r$.

There exists a sequence $C_2,C_3,\dots,C_h$ of cycles in $G$ such that:
\begin{align}\label{g818}
\text{for $2\le r \le h$,  there exists $e_{r-1}\in\ole_{r-1}$}&\text{ that is incident to $v_0$}\\
\text{and belongs to both $C_{r-1}$ }&\text{and $C_{r}$},
\nonumber \\ 
\lim_{n\to\oo} \si_{n,r-1}(0,E_2(C_{r}))^{1/n} &= \mu, \qquad  2 \le r < h,\label{g818c}\\
\liminf_{n\to\oo} \si_{n,h-1}(0,E_2(C_h))^{1/n} &< \mu,
\label{g818d}
\end{align}
where $\sigma_{n,r-1}(0,E_2(C_r))$ is the number of 
$n$-step SAWs on $G_{r-1}$ from $v_0$
for which $E_2(C_r)$ never occurs.
\end{lemma}

\begin{proof}[Proof of Lemma \ref{lem:Ch}]

We construct the $\ole_{r}$ and $C_r$ by iteration. 
Let $G_r'$ be obtained from $G_{r-1}$ (where $G_0:=G$)
by the deletion of all edges in $\ole_{r}$ incident to $v_0$.
Assume \eqref{g813} holds, and let $\ole_1 = \ol f_1$.

\smallskip
\noindent
\emph{The case $r=2$.}
By \eqref{g813}, for `most' $n$-step SAWs $\pi$ on $G$ from $v_0$,
and for all $j$ and all $\g C_1$ with $\g\in\olGa$ and 
$\g v_0=\pi_j$, $\pi_j$ is the unique vertex of $\g C_1$ visited by $\pi$. 
Since no edge in $\olF_1$ is traversed by
any path contributing to $\s_n(0,E^1_2)$, by \eqref{g813},   
\begin{equation}\label{g815}
\mu(G_1)=\mu.
\end{equation}

Let $v_0^{(1)}$ be the neighbour of $v_0$ such
that $f_1= \langle v_0, v_0^{(1)}\rangle$.
We claim  that, subject to \eqref{g813},
\begin{equation}\label{g814}
\text{$v_0$ and $v_0^{(1)}$ are connected in $G_1'$},
\end{equation}
and shall prove this by contradiction.
Assume \eqref{g813} holds, and that
$v_0$ and $v_0^{(1)}$ are not connected in $G_1'$.

Let $\pi^{(1)}$, $\pi^{(2)}$ be two SAWs from $v_0$ on $G_1$, and consider a walk
$\nu$ on $G$ given as follows: $\nu$ follows $\pi^{(1)}$ from $v_0$ to its other endvertex $z$, then traverses
an edge of the form $f =\langle z,w\rangle \in \ole_{1}$, 
and then follows $\g \pi^2$ for some $\g\in\olGa$ satisfying $\g v_0=w$. 
By the above assumption and vertex-transitivity, 
the removal of all edges of $\ole_1$ incident to $z$ disconnects $z$ and $w$, and thus
any such $\nu$ is a SAW on $G$. 

This process of concatenation may be iterated as follows.
Let $\pi^{(1)},\allowbreak
\pi^{(2)}, \dots,\pi^{(N)}$ be SAWs on $G_1$ from $v_0$. 
We aim to construct a SAW $\nu$ on $G$ as follows. Suppose we have concatenated $\pi^{(1)},\dots, \pi^{(r)}$ 
to obtain a SAW $\nu^{(r)}$ on $G$ from $v_0$ to some vertex $z$. 
We now traverse an edge of the form $f =\langle z,w\rangle\in \ole_{1}$, 
followed by the SAW $\g\pi^{(2)}$ 
for some $\g\in\olGa$ with $\g v_0=w$.
Unless (i)  there is a unique edge of $\ole_{1}$ incident to $z$  and (ii) $\pi^{(r)}$ is
the SAW with zero length, then there exists a choice of $w$ 
such that the resulting walk $\nu^{(r+1)}$
is a SAW.  
Therefore, if every $\pi^{(r)}$ has length $1$ or more,
the sequence $(\pi^{(r)})$ may be concatenated to obtain a SAW $\nu=\nu^{(N)}$ on $G$. 

Let $\si_{n,1}$ be the number of $n$-step SAWs from $v_0$ on $G_1$, and let $\mu_1 = \mu(G_1)$.
By  \eqref{1215},
\begin{equation}
\sigma_{2n,1}\geq \mu_1^{2n},\qquad n \ge 1. \label{cn1}
\end{equation}
Let $\Sigma^{\ro}_{2n}$ be the set of $2n$-step SAWs from $v_0$ on $G$ that traverse edges in $F_1$ 
(in either direction) at the
odd-numbered steps only, and write $\si^\ro_{2n} = |\Si^\ro_{2n}|$.
Our purpose in considering only the odd steps is to avoid component walks
of zero length.

Let $0 \le j \le n$, and consider the set of all $\pi\in\Si^\ro_{2n}$ 
that traverse exactly $j$ edges of $F_1$.
There are $\binom nj$ ways of choosing the indices of these edges within $\pi$.
By the construction described above, and \eqref{cn1}, there are at least
$\mu_1^{2n-j}$ choices for the sequence of SAWs on $G_1$ whose
concatenation makes $\pi$. Therefore,
\begin{equation}\label{g816}
\si_{2n}^\ro \ge \sum_{j=0}^n \binom nj \mu_1^{2n-j} = \mu_1^{2n}(1+\mu_1^{-1})^n.
\end{equation}
This implies
\begin{equation}\label{rcn}
\sigma_{2n} \geq \sigma^\ro_{2n}  \ge \mu_1^{2n}(1+\mu_1^{-1})^n,
\end{equation}
whence $\mu \ge \mu_1\sqrt{1+\mu_1^{-1}} > \mu_1$,
in contradiction of \eqref{g815}. Statement \eqref{g814} is proved.

By \eqref{g814}, $G$ has a shortest (directed) walk 
$W_2$ from $v_0$ to $v_0^{(1)}$
using no edge of $\ole_{1}$ incident to $v_0$. The walk $W_2$, followed by the 
(directed) edge $\langle  v_0^{(1)},v_0\rangle$, 
forms the required (directed) cycle $C_2$ of $G$, having $e_1$ 
in common with $C_1$. Let
$e_{2} = \langle v_0, v_0^{(2)}\rangle$ be the first edge of
$W_2$, 
and write $E_k^2= E_k(C_2)$ and so on. 

If
\begin{equation}\label{g817-}
\liminf_{n\to\oo} \si_{n,1}(0,E_2^2)^{1/n} < \mu,
\end{equation}
we have proved the claim with $h=2$.
Suppose conversely that
\begin{equation}\label{g817}
\lim_{n\to\oo} \si_{n,1}(0,E_2^2)^{1/n}  = \mu,
\end{equation}
whence $\mu(G_2)=\mu$ as in \eqref{g815}.

\smallskip
\noindent
\emph{The general case.}
Suppose $u \ge 2$, and cycles $C_2,C_3,\dots,C_u$ from $v_0$ 
have been found, with corresponding classes $\ole_2, \ole_3,\dots,\ole_u$
and edges $e_i = \langle v_0, v_0^{(i)}\rangle \in \ole_i$,  
such that \eqref{g818} holds with $h$ replaced by $u$, and 
\begin{equation}\label{g819}
\lim_{n\to\oo} \si_{n,{i-1}}(0,E_2^i)^{1/n} = \mu, \qquad  1 \le i \le u,
\end{equation}
where $E_2^i = E_2(C_i)$. By \eqref{g819} with $r=u$, we have as above that $\mu(G_u)=\mu$. 

Let $\sigma_{2n,{u-1}}^{\ro}$ be the number of $2n$-step SAWs 
on $G_{u-1}$ from $v_0$ that traverse edges in $F_u$
only at odd steps. If $v_0$ and $v_0^{(u)}$ are disconnected in  $G_u'$, we have
\begin{equation}\label{gi}
\sigma_{2n}\geq \sigma_{2n,{u-1}}^{\ro}
\geq \sum_{j=0}^{n}\binom{n}{j}\mu_u^{2n-j} =\mu_u^{2n}(1+\mu_u^{-1})^{n},
\end{equation}
where $\mu_u=\mu(G_u)$. 
This contradicts $\mu_u=\mu$, whence $v_0$ and $v_0^{(u)}$ are connected in $G_u'$.

Therefore,  $v_0$ and $v_0^{(u)}$ are connected by a shortest (directed) walk 
$W_{u+1}$ of $G_{u-1}$  using no edge of $\ole_u$ incident to $v_0$, and we write
$e_{u+1} = \langle v_0,v_0^{(u+1)}\rangle$ for the first edge of $W_{u+1}$,
and $C_{u+1}$ for the cycle in $G_{u-1}$ formed by $W_{u+1}$ followed by $\langle  v_0^{(u)},v_0\rangle$. 
If
\begin{equation}\label{g1111}
\liminf_{n\to\oo} \si_{n,u}(0,E_2^{u+1})^{1/n} < \mu, 
\end{equation}
we stop and set $h=u+1$, and otherwise we continue as above.

\smallskip

Either this iterative process terminates at some earliest $u$ satisfying \eqref{g1111}, or not.
In the step from $G_u$ to $G_{u+1}$, the degree of $v_0$ is
reduced  strictly.
If the process does not terminate in the manner of \eqref{g1111}, 
all edges incident to $v_0$ in the last 
non-trivial graph encountered, $G_t$ say, lie in
the same  equivalence class. Furthermore,
the equation of \eqref{g819} holds with $i$ replaced by $t+1$, and hence $\mu_{t+1}=\mu$.
This is impossible since $v_0$ is isolated in $G_{t+1}$.
The proof is complete.
\end{proof}

If \eqref{g813} fails, we set $h=1$ and $G_0=G$. 
If \eqref{g813} holds, we adopt the notation of Lemma \ref{l1}, and write
$E_k^i = E_k(C_i)$ and so on. 
For an $n$-step  SAW $\pi=(\pi_0,\pi_1,\dots,\pi_n)$ on $G$, 
we say that $E_k^i(m)$ occurs at the $j$th step of $\pi$ if 
$E_k^i$ occurs at the 
$m$th step of the $2m$-step SAW $(\pi_{j-m},\dots,\pi_{j+m})$
(with a modified definition if either $j - m<0$ or $j+m>n$, as in the second paragraph after \eqref{200}). 

\begin{lemma}\label{l1}

Let $1 \le i \le h$ and $1\le k \le |C_i|$.
If
$$
\liminf_{n\to\oo} \si_{n,{i-1}}(0,E^i_k) < \mu,
$$
then there exist $a>0$ and $m\in\NN$ such that
\[
\limsup_{n\to\infty} \sigma_{n,{i-1}}(an,E_k^i(m))^{1/n}<\mu.
\]
\end{lemma}

\begin{proof}
This is proved in the same manner as  was Lemma \ref{zl}.
\end{proof}

\begin{proof}[Proof of Lemma \ref{l2}]

Assume the converse of \eqref{g8120}, namely \eqref{g813},
and adopt the notation of Lemma \ref{lem:Ch}. Recall
the notation $E_*$ from the beginning of Section \ref{sec:cayley-proof}.
Exactly
one of the following holds:
\begin{align}
\liminf_{n\to\infty} \sigma_{n,{h-1}}(0,E^h_*)^{1/n}<\mu,\label{ca}\\
\lim_{n\to\infty} \sigma_{n,{h-1}}(0,E^h_*) ^{1/n}=\mu,\label{cb}
\end{align}
and we shall derive a contradiction in each case.

Assume first that \eqref{ca} holds. No SAW counted in $\si_{n,h-2}(0,E_2^{h-1})$ traverses 
any edge of $\ole_{h-1}$, whence
$\sigma_{n,{h-2}}(0,E^{h-1}_2)=\si_{n,{h-1}}(0,E^{h-1}_2)$.
Since $C_h$ and $C_{h-1}$ 
have at least one common edge, we deduce that 
$$
\sigma_{n,{h-2}}(0,E^{h-1}_2)\le \si_{n,h-1}(0,E^h_*),
$$
and hence
\begin{equation*}
\liminf_{n\to\infty} \sigma_{n,{h-2}}(0,E_2^{h-1})^{{1}/n}<\mu,
\end{equation*}
in contradiction of the minimality of $h$. 

Assume now that \eqref{cb} holds. By \eqref{g818d} 
and the monotonicity in $k$ of $\sigma_{n,{h-1}}(0,E_k^h)$, 
there exists $k\in\NN$ satisfying $2\leq k<\lell := |C_h|$ such that
\begin{align}
\liminf_{n\to\infty} \sigma_{n,{h-1}}(0,E_k^h)^{1/n}&<\mu,\label{ie-}\\
\lim_{n\to\infty} \sigma_{n,{h-1}}(0,E_{k+1}^h)^{1/n} &=\mu.\label{ee}
\end{align}
By \eqref{ie-} and Lemma \ref{l1}, there exist $a>0$ and $m\in\NN$ such that
\begin{equation}
\limsup_{n\to\infty} \s_{n,{h-1}}(an, E_k^h(m))^{1/n}<\mu.\label{ie}
\end{equation}
We shall derive a contradiction from \eqref{ee}--\eqref{ie}
in a similar manner to the proof of \eqref{g11055}.

Let $T_n$ be the set of  $n$-step SAWs on $G_{h-1}$ from $v_0$
such that $E_{k+1}^h$ never occurs, and 
$E_k^h(m)$ occurs at least $an$ times. By \eqref{ee}--\eqref{ie},  
\begin{equation}\label{g1003}
\lim_{n\to\infty} |T_n|^{1/n}=\mu.
\end{equation}

Let $\pi=(\pi_0,\pi_1,\dots,\pi_n)  \in T_n$ and  
$u=\lfloor\kappa n\rfloor -2$ where
$$
\kappa =\frac{a}{(2m+2)\ell}.
$$ 
As after \eqref{g459}, we may find $j_1< j_2< \dots< j_u$ and
$\g_1,\dots,\g_u\in \olGa$ such that
\begin{gather}
\g_t(v_0)=\pi_{j_t},\qquad 1 \le t \le  u,\\
\text{each $2m$-step SAW $(\pi_{{j_t}-m},\dots,\pi_{{j_t}+m})$ visits}\label{g820-3}\\
\text{at least $k$ vertices of $\g_t C_h$},\nonumber\\ 
0<j_1-m,\q  j_u+m<n,\label{g820-2}\\
j_t+m<j_{t+1}-m, \qquad 1 \le t < u,\label{g820-1}\\
\g_1C_h,\g_2 C_h,\dots,\g_u C_h
\text{ are pairwise vertex-disjoint}.\label{g820}
\end{gather}
For $t=1,2,\dots,u$, let
\begin{equation}\label{g820+1}
\alpha_{t}=\min\{i: \pi_i \in \g_t C_h\},\q
\omega_{t}=\max\{i: \pi_i \in \g_t C_h\}.
\end{equation}

Since $E_k^h(m)$ occurs at the $j_t$th step but not $E_{k+1}^h$, there 
are exactly $k$ points of $\g_t C_h$ that are visited by
$\pi$, and these points lie on $\pi$ between positions $j_t-m$ and $j_t+m$. 
Therefore, 
\begin{equation}\label{g820+2}
j_t-m\leq \alpha_t<\omega_t\leq j_t+m, \qquad 1 \le t \le u.
\end{equation}
We propose to replace the subwalk $(\pi_{\a_t},\dots,\pi_{\om_t})$ 
by the part of the cycle $\g_t C_h$ with the same endpoints and using at least one edge in 
$\g_t C_{h-1}$; this may be done since $\g_t C_{h}$ and $\g_t C_{h-1}$ 
have at one edge in common,
namely $\g_t e_{h-1}$. 
By \eqref{g820}, such a replacement may be performed simultaneously
for all $t$. The resulting walk $\psi$ is 
a SAW on $G$  with length $n'$ satisfying
\begin{equation*}
n'<n+u\ell.
\end{equation*}
Furthermore, since $\pi\in T_n$,  the only edges of $\ole_{h-1}$
in $\psi$ are those introduced during a
substitution. 

Let $\de>0$ and $s=\de n$, where $\de$ will be chosen later
(and we omit integer-part symbols as before).
Consider the set of pairs $(\pi, H)$ where $\pi\in T_n$, 
and $H=(h_1,h_2,\dots,h_s)$ is an ordered 
subset of $\{j_1,j_2,\dots,j_u\}$.  We may
make the above replacement around each $\pi_{h_i}$ to obtain a SAW
$\psi=\psi(\pi,H)$ on $G$

As in \eqref{lb0},
\begin{equation}
|(\pi,H)| \geq |T_n|\binom{\kappa n-2}{\delta n}.  \label{lb}
\end{equation}
For an upper bound, consider a given pair $(\pi, H)$. 
Edges in $\ole_{h-1}$ are traversed
between $|H|=\de n$ and $\ell\delta n$ times on $\psi$. 
Therefore, given $\psi$, there are at 
most $2m\binom{\ell\delta n}{\delta n}$ ($\le m 2^{\ell\de n+1}$)
possibilities for 
the location of the earliest point of $\psi$ in $\g_t C_h$. 
Given $\psi$ and the locations of  these earliest points,
there are at most $\De^{\ell}$ different choices for 
each $\g_t C_h$. Once these are determined, 
each such $\g_t C_h$ determines a subwalk of 
$\psi$ that replaces some subwalk of $\pi$. Since each of the replaced subwalks of 
$\pi$ has length not exceeding $2m$, 
there are at most $Y^{\de N}$ possibilities for $\pi$, where $Y= \sum_{i=0}^{2m}\si_i$. 
Therefore, 
\begin{equation}\label{up}
|(\pi,H)| \leq m2^{\ell\delta n+1}\De^{\ell}Y^{\delta n}\sum_{i=0}^{n+\ell\delta n}\si_i.
\end{equation}

We  combine \eqref{lb} and \eqref{up}, take the $n$th root and let $n\to\infty$, obtaining
by \eqref{g1003} that
\begin{equation*}
\mu\frac{\kappa^{\kappa}}{\delta^{\delta}(\kappa-\delta)^{\kappa-\delta}}
\leq 2^{\ell\delta}Y^{\delta}\mu^{1+\ell\delta}.
\end{equation*}
Setting $Z=(2\mu)^{\ell}Y$ and $\eta={\delta}/{\kappa}$, we deduce that
\begin{equation*}
1\leq [Z^\eta\eta^\eta(1-\eta)^{1-\eta}]^{\kappa}.
\end{equation*}
As at the end of the proof of Lemma \ref{zll},
this is a contradiction for small $\eta>0$.

In conclusion, if \eqref{g813} and \eqref{cb} hold, we have a contradiction,
and the lemma is proved.
\end{proof}

\begin{proof}[Proof of Theorem \ref{thm:addedge}]

Write $\olmu := \mu(\olG)$, and let $F$ be the set of edges of $\olG$ not in $G$. 
Let $f$ and $C_1$ be given as around \eqref{ccl}, and let $\ell=\modl +1$ be the number of vertices
in $C_1$ (or, equivalently, the number of edges in $C_1$ viewed
as a cycle of $\olG$).
It suffices to make the following assumption.

\begin{assumption}\label{cond}
We have that $F=\osA f$.
\end{assumption}

For $v,w\in V$ and a graph $H$ with vertex-set $V$,  let
$N(v,w;H)$ denote the number of edges
between $v$ and $w$ in  $H$.
For $\g \in \sA$, we construct the graph denoted $\g \olG$  as follows. 
First, $\g \olG$ has vertex-set $V$.  
For $v,w \in V$, we place $N(v,w;\olG)$ edges between
$\g v$ and $\g w$ in $\g\olG$. Thus, $\g \olG$ is
obtained from $\g G$ by adding the edges of $\g F$.
Two elementary properties of $\g\olG$ are as follows.

\begin{lemma}\label{lem:abar}
Let $\g \in \sA$.
\begin{letlist}
\item $\g \olG$ is isomorphic to $\olG$.
\item $\g\osA\g^{-1} \subseteq \Aut(\g \olG)$.
\end{letlist}
\end{lemma}

\begin{proof}
(a) This is an immediate consequence of the definition of $\g \olG$.

\smallskip
\noindent
(b) By the definition of $\g\olG$, for $\a\in\osA$,
\begin{align*}
N\bigl(\g\a\g^{-1}(\g v), \g\a\g^{-1}(\g w); \g \olG\bigr)
&= N(\g\a v, \g\a w;\g\olG)\\
&=N(\a v, \a w; \olG)\\
&= N(v,w;\olG)\\
&= N(\g v, \g w; \g\olG),
\end{align*}
and the claim follows.
\end{proof}

Assume $\olmu=\mu$.
By Lemma \ref{l2} with $E_k := E_k^1$,
\begin{equation}
\liminf_{n\to\infty}\sigma_n(0,E_2)^{1/n}<
\mu\ (=\olmu).  \label{as}
\end{equation}
Exactly one of the following holds:
\begin{align}
\liminf_{n\to\infty}\sigma_n(0,E_\ell)^{1/n}&<\mu,\label{bil}\\
\lim_{n\to\infty}\sigma_n(0,E_\ell)^{1/n}&=\mu.\label{bie}
\end{align}

Assume first that \eqref{bil} holds. By Lemma \ref{l1}, there exist $a>0$ and $m\in \NN$
such that 
\begin{equation}
\limsup_{n\to\infty}\sigma_n(an, E_{\ell}(m))^{1/n}<\mu. \label{ccdd}
\end{equation}
Let $R_n$ be the set of $n$-step SAWs from $v_0$ on $G$ 
on which $E_{\ell}(m)$ occurs at least $an$ times.
By \eqref{ccdd},
\begin{equation}\label{ce}
\lim_{n\to\infty} |R_n|^{1/n}=\mu.
\end{equation}

Let $\rho \in V$ be such that $\osA$ has the \fcp\ with root $\rho$, 
and let $\nu_0,\nu_1,\dots,\nu_s$
be given accordingly as in Definition \ref{def:fcp}.
Let $\sO=\sA C_1$ be the orbit under $\sA$ of $C_1$ 
viewed as a set of labelled vertices, and let $\sO_i = \nu_i \osA C_1$.
Let $0 \le i \le s$. We say that $E_{\ell}^i(m)$ 
occurs at the $j$th step of a SAW $\pi$ on $G$ if there exists $\g\in\sA$
with $\g \pi_0=\pi_j$ and $\g C_1\in\sO_i$,
such that all the vertices of $\g C_1$ are visited by the $2m$-step SAW $(\pi_{j-m},\dots,\pi_{j+m})$ (subject to the usual amendment if
$j-m<0$ or $j +m >n$).

Let $R_{n}^{(i)}$ be the set of $n$-step SAWs from $v_0$ on $G$ 
for which $E_{\ell}^{i}(m)$
occurs at least ${an}/{(s+1)}$ times. Thus,
\begin{equation}\label{un}
R_n\subseteq\bigcup_{i=0}^{s}R_n^{(i)}.
\end{equation}
By \eqref{ce}--\eqref{un}, there exists $i$ satisfying $0\leq i\leq s$ such that
\begin{equation}\label{un2}
\limsup_{n\to\infty} |R_n^{(i)}|^{1/n}=\mu,
\end{equation}
and we choose $i$ accordingly.

We now apply the argument in the proof of Lemma \ref{l2}
with the constant $a$ replaced by $a/(s+1)$.
Let $\pi\in R_n^{(i)}$ be such that $E_{\ell}^{i}(m)$ occurs at 
steps $j_1<j_2<\dots<j_u$, and in addition
\eqref{g820-2}--\eqref{g820} hold with $h$ replaced by $1$,
and with each $\g_t C_1 \in \sO_i$.
Since $\g_t C_1 \in \sO_i = \nu_i \osA C_1$, there exists
$\a_t \in\osA$ such that $\g_t C_1 = [\nu_i \a_t \nu_i^{-1}] \nu_iC_1$.
By Lemma \ref{lem:abar}(b),   $[\nu_i \a_t \nu_i^{-1}] \nu_i \olG$
is isomorphic to $\nu_i\olG$. Therefore, $\g_t f$ is an edge
of $\nu_i \olG$ but not of $\nu_i G$.
We think of $\pi$ as a SAW on the graph $\nu_i G$.

Let $\a_t$, $\om_t$ be as in \eqref{g820+1}, and note that \eqref{g820+2}
holds as before.
Consider the replacement of the 
subwalk $(\pi_{\a_t},\dots,\pi_{\omega_t})$ 
by a walk that goes along that part of $\g_t C_1$ ($=  \g_t'\nu_i C_1$ where
$\g_t' = \nu_i \a_t \nu_i^{-1}$),
viewed as a cycle,  
that includes an edge of $\g_t F$ ($= \g'_t\nu_i F$). 
By Lemma \ref{lem:abar}(b), the new walk is a SAW on $\nu_i\olG$.
Furthermore it uses edges of $\nu_i\olG$ 
not belonging to $\nu_i G$. Lower and upper bounds 
may be derived as in the proof of Lemma \ref{l2},
and these lead to a contradiction when working along the 
subsequence implicit in \eqref{un2}. 
This implies $\mu(G)<\mu(\nu_i\olG )$ subject
to \eqref{bil}. 
By Lemma \ref{lem:abar}(a), $\mu(\nu_i\olG)=\olmu$ and
hence $\mu(G) < \olmu$.

Assume finally that \eqref{bie} holds. By \eqref{as}--\eqref{bie}, 
there exists $k\in\NN$ satisfying $2\leq k< \ell$ such that
\begin{align}
\liminf_{n\to\infty}\sigma_n(0,E_k)^{1/n}&<\mu, \nonumber\\
\lim_{n\to\infty}\sigma_n(0,E_{k+1})^{1/n}&=\mu. \label{cdn1}
\end{align}
By Lemma \ref{l1}, there exist $a>0$ and $m\in\NN$ such that
\begin{equation}
\limsup_{n\to\infty}\sigma_n(an,E_k(m))^{1/n}
<\mu.\label{cdn2}
\end{equation}

Let  $S_n$ be the set of $n$-step SAWs from $v_0$ on $G$ such that
$E_{k+1}$ never 
occurs  and  $E_k(m)$ occurs at least $an$  times. 
By \eqref{cdn1}--\eqref{cdn2},
\begin{equation}\label{ccq}
\lim_{n\to\infty} |S_n|^{1/n}=\mu.
\end{equation}
For $0 \le i \le s$, let $S_n^{(i)}$ be the set of $n$-step SAWs from 
$v_0$ on $G$ such that $E_{k+1}$ never occurs,
and $E_k^i(m)$ occurs at least $an/s$ times. Thus,
\begin{equation}\label{cun}
S_n\subseteq\bigcup_{i=0}^{s} S_n^{(i)}.
\end{equation}
By \eqref{ccq}--\eqref{cun}, there exists $i$ satisfying $0\leq i\leq s$ such that
\begin{equation*}
\limsup_{n\to\infty}|S_n^{(i)}|^{1/n}=\mu,
\end{equation*}
and we choose $i$ accordingly.

We apply the argument of the proof of Lemma \ref{l2} once again.
Let $\pi\in S_n^{(i)}$ be such that $E^i_k(m)$ occurs at steps $j_1<j_2<\dots<j_u$, 
and in addition
\eqref{g820-2}--\eqref{g820} hold with $\g_t C_1 \in \sO_i$.
The argument is as above, and we do not repeat the details. 
This implies $\mu(G)<\mu(\olG)$ when \eqref{bie} holds, and the proof is complete.
\end{proof}

\section*{Acknowledgements} 
This work was supported in part by the Engineering
and Physical Sciences Research Council under grant EP/103372X/1.
It was completed during a visit by GRG to the Theory Group
at Microsoft Research.
The authors thank Alexander Holroyd and Yuval Peres for 
several helpful discussions
concerning the algebra of graphs, and especially the former for his suggestions
concerning Theorem \ref{thm:addedge}.

\bibliography{saw-ineq2}
\bibliographystyle{amsplain}

\end{document}